\documentclass[10pt,english,reqno]{amsart} 
\usepackage[mathscr]{euscript}
\newcommand{\fr}{\mathfrak}
\newcommand{\cal}{\mathscr}

\newcommand{\op}{\operatorname}
\newcommand{\tn}{\textnormal}

\newcommand{\Spec}{\mathrm{Spec}}

\newcommand{\Ran}{\mathrm{Ran}}

\newtheorem{thmx}{Theorem}

\newtheorem{corx}[thmx]{Corollary}

\newtheorem{thm}{Theorem}[section]

\newtheorem{prop}[thm]{Proposition}
\newtheorem{lem}[thm]{Lemma}

\theoremstyle{definition}

\newtheorem{claim}[thm]{Claim}

\newtheorem{eg-eg}[thm]{Example of Example}

\newtheorem{rem}[thm]{Remark}

\usepackage[colorlinks=true]{hyperref}
\usepackage{graphicx}
\usepackage{amssymb}
\usepackage{epstopdf}
\usepackage{enumerate}
\usepackage{tikz}
\usepackage{marginnote}

\usepackage{soul}

\usepackage{mdframed}
\newmdenv[
  topline=false,
  bottomline=false,
  rightline=false,
  skipabove=\topsep,
  skipbelow=\topsep
]{siderules}

\numberwithin{equation}{section}

\usepackage{mathtools}

\usepackage{epigraph}



\usepackage{color}
\usepackage[color,matrix,arrow]{xy}
\xyoption{all}

\title{Extensions by $\mathbf K_2$ and factorization line bundles}

\author{James Tao \and Yifei Zhao}

\email{jamestao@mit.edu}
\email{yifei@math.harvard.edu}

\begin{document}

\begin{abstract}
Let $X$ be a smooth, geometrically connected curve over a perfect field $k$. Given a connected, reductive group $G$, we prove that central extensions of $G$ by the sheaf $\mathbf K_2$ on the big Zariski site of $X$, studied in Brylinski-Deligne \cite{BD01}, are equivalent to factorization line bundles on the Beilinson-Drinfeld affine Grassmannian $\op{Gr}_G$. Our result affirms a conjecture of Gaitsgory-Lysenko \cite{GL16} and classifies factorization line bundles on $\op{Gr}_G$.
\end{abstract}

\maketitle

\setcounter{tocdepth}{1}
\tableofcontents

\section*{Introduction}

This paper compares two kinds of data parametrizing metaplectic extensions of the Langlands program. One is $\tn K$-theoretic, and the other has to do with factorization structures on the affine Grassmannian $\op{Gr}_G$.

\smallskip

Let us first explain how these structures arise in the theory.

\subsection{$\tn K$-theoretic metaplectic parameters}

\subsubsection{} In the classical theory of automorphic forms, one starts with a global field $\mathbf F$ and a reductive group $G$ over it. Denote by $\mathbb A_{\mathbf F}$ the topological ring of ad\`eles of $\mathbf F$. The principal objects of interest are certain functions on the homogeneous space $G(\mathbb A_{\mathbf F})/G(\mathbf F)$. Roughly speaking, the goal of the Langlands program is to relate them to representations of $\tn{Gal}(\overline{\mathbf F}/\mathbf F)$ valued in the $\tn L$-group of $G$.

\subsubsection{} The study of automorphic forms has seen several generalizations, where one replaces $G(\mathbb A_{\mathbf F})$ by certain topological coverings. The first example of such a covering is the metaplectic group constructed by Weil \cite{We64}. These are double covers of the symplectic groups $\tn{Sp}_{2n}(\mathbf F_{\nu})$, for local fields $\mathbf F_{\nu}$, and combine into a cover of $\tn{Sp}_{2n}(\mathbb A_{\mathbf F})$ equipped with a section over $\tn{Sp}_{2n}(\mathbf F)$.

\subsubsection{}
The existence of interesting topological coverings is by no means restricted to the symplectic group. For any reductive group $G$, Brylinski--Deligne \cite{BD01} observed that a large class of coverings of $G(\mathbb A_{\mathbf F})$ arise from $\tn K$-theoretic data.

\smallskip

To explain their work more precisely, we let $\mathbf K_2$ denote the Zariski sheafification of the second algebraic $\tn K$-theory group. Brylinski--Deligne \cite{BD01} started with a central extension:
\begin{equation}
\label{eq-bd-data-fractional}
1 \rightarrow \mathbf K_2 \rightarrow E \rightarrow G \rightarrow 1
\end{equation}
of sheaves on the big Zariski site of $\mathbf F$. Using the Hilbert symbol on local fields $\mathbf F_{\nu}$, they produced topological central extensions $\widetilde G$ of $G(\mathbb A_{\mathbf F})$ by the group $\mu_{\mathbf F}$ of roots of unity in $\mathbf F$. As a consequence of the reciprocity law of the Hilbert symbol, the central extension $\widetilde G$ splits over $G(\mathbf F)$ \cite[\S10]{BD01}. This splitting makes it possible to define ``metaplectic'' automorphic forms as functions on $\widetilde G/ G(\mathbf F)$, equivariant against a character of $\mu_{\mathbf F}$ and satisfying certain analytic properties.

\subsubsection{} The main theorem of \emph{loc.cit.}~is that the groupoid of central extensions \eqref{eq-bd-data-fractional} admits a purely combinatorial description. Among other things, the Brylinski--Deligne theorem allows one to define the $\tn L$-group associated to such a central extension, as has been done by Weissman \cite{We15}. These works bring the study of metaplectic automorphic forms into the broader scope of the Langlands program.

\smallskip

It is thus reasonable to view central extensions by $\mathbf K_2$ as metaplectic parameters of the Langlands program and the resulting topological coverings $\widetilde G$ as ``metaplectic groups'' for $G(\mathbb A_{\mathbf F})$.

\subsection{Geometric metaplectic parameters}

\subsubsection{} Let us now specialize to the function field case, where a more geometric perspective in generalizing the Langlands program is available.

\smallskip

We fix a finite ground field $k$ and a smooth, proper, geometrically connected curve $X$ over $k$. The letter $\mathbf F$ will stand for the field of fractions of $X$. For simplicity, let us also assume that the reductive group $G$ is defined over $k$. In the function field setting, automorphic functions can be accessed via $\ell$-adic sheaves on the moduli stack $\tn{Bun}_G$ of principal $G$-bundles on $X$ (or certain variants thereof in the ramified situation).

\subsubsection{} The moduli stack $\tn{Bun}_G$ has a local avatar, known as the affine Grassmannian $\tn{Gr}_G$. It is an ind-scheme attached to each closed point $x$ of $X$ and comes equipped with a canonical map to $\tn{Bun}_G$. Unlike $\tn{Bun}_G$, the affine Grassmannian is not in general (ind-)smooth, and its singularities carry representation-theoretic meaning.

\smallskip

Following Beilinson--Drinfeld \cite{BD04}, one may also view $\tn{Gr}_G$ as an ind-scheme over the global curve $X$. As such, it possess an additional structure called ``factorization.'' Roughly speaking, factorization structure describes a fusion rule of the fibers of $\tn{Gr}_G$ as distinct points merge in $X$. A precise formulation makes use of the Ran space and will be recalled in \S\ref{sec-notations}.

\subsubsection{} The ind-scheme $\tn{Gr}_G$ plays a central role in the Langlands program. Indeed, the category of spherical sheaves $\tn{Sph}_G$, which are $\overline{\mathbb Q}_{\ell}$-sheaves on $\tn{Gr}_G$ equivariant with respect to the arc group (i.e., positive loop group), is equivalent to representations of the Langlands dual group $\check G$ under the geometric Satake equivalence of Mirkovi\'c--Vilonen \cite{MV07}. Their proof uses the factorization structure of $\tn{Gr}_G$ to construct the symmetry constraint for the convolution monoidal structure on $\tn{Sph}_G$.

\subsubsection{} Motivated by these considerations, Gaitsgory--Lysenko \cite{GL16} proposed to define \emph{geometric metaplectic parameters} as \'etale gerbes over $\tn{Gr}_G$ banded by a suitable torsion abelian group $A\subset\overline{\mathbb Q}_{\ell}^{\times}$, which furthermore respect its factorization structure. These objects are called \emph{factorization gerbes}.

\smallskip

A factorization gerbe allows one to form a ``twisted'' (or ``metaplectic'') category of $\ell$-adic sheaves on $\tn{Gr}_G$. Furthermore, it is possible to replicate the Mirkovi\'c--Vilonen proof in order to construct the metaplectic geometric Satake equivalence and formulate a vanishing conjecture in the metaplectic geometric Langlands program \cite[\S9-10]{GL16}.

\subsection{Relationship between the two}

\subsubsection{}
Let us now turn to the problem addressed in the present paper, which is a comparison of the above two kinds of metaplectic parameters. Since the problem is independent of the global geometry of $X$, we shall formulate it for any smooth, geometric connected curve $X$ (i.e., \emph{not} necessarily proper) over a perfect ground field $k$.

\smallskip
 
We shall consider the groupoid of central extensions of $G$ by $\mathbf K_2$, over the big Zariski site of $X$ rather than its field of fractions\footnote{Any central extension over $\mathbf F$ extends to one over $X_1$ for some open $X_1\subset X$. Any two such extensions to $X_1$ become canonically isomorphic over some open $X_2\subset X_1$.}. We denote this groupoid by $\mathbf{CExt}(G, \mathbf K_2)$.

\subsubsection{} Subject to a restriction on $\tn{char}(k)$, Gaitsgory \cite{Ga18} defined a functor from the groupoid $\mathbf{CExt}(G, \mathbf K_2)$ to the (2-)category $\mathbf{Ge}_{A}^{\tn{fact}}(\tn{Gr}_G)$ of factorization gerbes on $\tn{Gr}_G$. It is a composition of two functors:
\begin{align*}
\mathbf{CExt}(G, \mathbf K_2) \xrightarrow{\Phi_G} &\; \mathbf{Pic}^{\tn{fact}}(\tn{Gr}_G) \\
& \xrightarrow{\tn{Kum.}} \mathbf{Ge}^{\tn{fact}}_{A}(\tn{Gr}_G).
\end{align*}
Here, $\mathbf{Pic}^{\tn{fact}}(\tn{Gr}_G)$ denotes the groupoid of factorization line bundles, i.e., line bundles on $\tn{Gr}_G$ which respect its factorization structure. The second functor is a standard construction using the Kummer exact sequence (where we fix an element in $A(-1)$.) The first functor $\Phi_G$, a kind of residue map on algebraic $\tn K$-theory, is more interesting. To wit, it relates $\tn K$-theoretic data to purely geometric objects. The comparison of the two kinds of metaplectic parameters thus amounts to understanding the behavior of $\Phi_G$.

\smallskip

The restriction on $\tn{char}(k)$ enters in the definition of $\Phi_G$---it states that $\tn{char}(k)$ cannot divide a certain integer $N_G$ which depends on $G$. The integer $N_G$ is the index of the subgroup of the group of Weyl-invariant, integral quadratic forms on the co-weight lattice generated by Chern classes of representations $G \rightarrow \tn{GL}(V)$ (see \cite[\S0.1.8]{Ga18}). Tautologically, the condition $\tn{char}(k)\nmid N_G$ is satisfied when $G$ is a product of general linear groups (or when $\tn{char}(k) = 0$).

\subsubsection{Main result} We can now state our main result, which asserts that $\Phi_G$ is an equivalence of categories whenever it is defined. It will appear as Theorem \ref{thm-main} in the main text.

\begin{thmx}
\label{thmx-main}
Suppose $k$ is a perfect field, $X$ is a smooth, geometrically connected curve and $G$ is a connected reductive group over $k$. If $\tn{char}(k)\nmid N_G$, then $\Phi_G$ is an equivalence:
$$
\Phi_G : \mathbf{CExt}(G, \mathbf K_2) \xrightarrow{\sim} \mathbf{Pic}^{\op{fact}}(\op{Gr}_G).
$$
\end{thmx}

\noindent
This result affirms \cite[Conjecture 3.4.2]{GL16}. Roughly speaking, it means that no information is lost when we pass from $\tn K$-theoretic metaplectic data to geometry of the affine Grassmannian.

\subsubsection{} Let us note some consequences of Theorem \ref{thmx-main}. First, the classification theorem of Brylinski--Deligne \cite{BD01} applies to any regular scheme of finite type over a field. In particular, $\mathbf{CExt}(G, \mathbf K_2)$ is equivalent to a groupoid of combinatorial gadgets, to be denoted by $\theta_G(\Lambda_T)$.

\smallskip

We shall establish a commutative triangle (appearing as \eqref{eq-compatible-main} in the main text):
\begin{equation}
\label{eq-main-triangle}
\xymatrix@C=0em@R=1.5em{
	\mathbf{CExt}(G, \mathbf K_2) \ar[rr]^-{\Phi_G}\ar[rd]_{\Phi_{\op{BD}}} & & \mathbf{Pic}^{\op{fact}}(\op{Gr}_G) \ar[dl]^{\Psi} \\
	& \theta_G(\Lambda_T)
}
\end{equation}
where the functor $\Psi$ is defined in explicit terms (i.e., without recourse to algebraic $\tn K$-theory). Therefore, Theorem \ref{thmx-main} implies a combinatorial classification of factorization line bundles on $\tn{Gr}_G$. The notation $\theta_G(\Lambda_T)$ is meant to recall the groupoid of ``$\theta$-data'' considered by Beilinson--Drinfeld \cite{BD04}, whose classification of factorization line bundles on the space of colored divisors is a precursor to our theorem.

\subsubsection{}
Another application of our theorem is the following.

\begin{corx}
\label{corx-bun}
Suppose we are under the hypothesis of Theorem \ref{thmx-main} and $X$ is furthermore proper. Then every factorization line bundle on $\op{Gr}_G$ canonically descends to $\op{Bun}_G$.
\end{corx}

\noindent
Indeed, this follows from the fact that the composition:
$$
\mathbf{CExt}(G, \mathbf K_2) \xrightarrow{\Phi_G} \mathbf{Pic}^{\op{fact}}(\op{Gr}_G) \rightarrow\mathbf{Pic}(\op{Gr}_G)
$$
factors through $\mathbf{Pic}(\op{Bun}_G)$ (see \cite[\S2.4]{Ga18}). Our Corollary may be viewed as an analogue of Gaitsgory's theorem \cite{Ga13} on cohomological contractibility of the fibers of $\tn{Gr}_G \rightarrow \tn{Bun}_G$.

\subsection{Our strategy}

\subsubsection{} We should mention first that our proof of Theorem \ref{thmx-main} relies on the classification theorem of Brylinski--Deligne, a fact which has two practical implications:
\begin{enumerate}[(a)]
	\item One does \emph{not} need to know the precise definition of $\Phi_G$ in order to understand our proof; in fact, as long as $\Phi_G$ gives the correct value on regular test schemes $S\rightarrow\op{Gr}_G$ (where it is defined using Gersten's resolution of $\mathbf K_2$) and satisfies some reasonable properties, then our proof runs through.
	\item After all functors in the triangle \eqref{eq-main-triangle} are defined, checking that it commutes is an essential step towards the proof, and takes up a large part of our work.
\end{enumerate}

A proof of Theorem \ref{thmx-main} without using the Brylinski--Deligne classification would certainly be desirable, but the authors could not find one.\footnote{As of now, even the definition of $\Phi_G$ appeals to the Brylinski--Deligne classification (\cite[\S5.1]{Ga18}).}

\subsubsection{} Assuming the commutativity of \eqref{eq-main-triangle} (which will be proved in \S\ref{sec-compatible}), our proof of the main theorem proceeds by checking that $\Psi$ is an equivalence for various kinds of reductive groups $G$. We summarize the key ideas and make attributions below (although the main text is organized somewhat differently):

\smallskip

\textbf{Step 1:} $G=T$ is a (split) torus. This case amounts to showing that $\mathbf{Pic}^{\op{fact}}(\op{Gr}_T)$ is equivalent to $\theta$-data for the lattice $\Lambda_T$. This is the content of \S\ref{sec-torus}. In fact, we will show that the same is true for factorization line bundles on various versions of $\op{Gr}_T$. This part of the proof relies on A.~Beilinson and V.~Drinfeld's classification of factorization line bundles on $\Lambda_T$-colored divisors of $X$ (see \cite{BD04}) and the Pic-contractibility of $\Ran(X)$ (\cite{Ta19}).

\smallskip

\textbf{Step 2:} $G$ is semisimple and simply connected. This case is essentially reduced to classifying line bundles on $\op{Gr}_G$ at a point of the curve $X$, and the latter has been worked out by G.~Faltings \cite{Fa03}. Since this case is also needed in proving the commutativity of \eqref{eq-main-triangle}, it will appear along with it in \S\ref{sec-compatible}.

\smallskip

\textbf{Step 3:} The derived subgroup $G_{\op{der}}$ is simply connected. This case essentially follows from the two previous ones. More precisely, let $T_1$ be the torus $G/G_{\op{der}}$. We observe that $\op{Gr}_G$ is an \'etale-locally trivial fiber bundle over $\op{Gr}_{T_1}$, with typical fiber $\op{Gr}_{G_{\op{der}}}$. We then use our knowledge from Step 2 to study when a factorization line bundle on $\op{Gr}_G$ descends to $\op{Gr}_{T_1}$, and we use Step 1 to classify the ones that are pulled back from the base.

\smallskip

\textbf{Step 4:} An arbitrary reductive group $G$. This follows from the previous cases, by $h$-descent of line bundles on derived schemes.\footnote{Aside from this descent technique, which was suggested to us by D.~Gaitsgory, our paper lives entirely within classical (i.e.,~non-derived) algebraic geometry.} Steps 3 and 4 form the content of \S\ref{sec-main}.

\subsection{Notations and conventions}
\label{sec-notations}

\subsubsection{} Unlike the main references \cite{GL16} \cite{Ga18}, we do \emph{not} need the theory of $\infty$-categories. Hence terms such as categories, groupoids, prestacks, etc., are understood in the classical sense.

\smallskip

Moreover, the prestacks we consider in the present paper are $0$-truncated. Namely, they are synonymous to presheaves on the category of affine schemes. However, in order to stay consistent with existing literature, we shall continue to call them prestacks.

\subsubsection{} Throughout the paper, we let $k$ be an \emph{algebraically closed} field. The general case of a perfect field is handled using Galois descent. The fact that central extensions by $\mathbf K_2$ satisfy Galois descent follows from work of Colliot-Th\'el\`ene and Suslin. We refer the reader to \cite[\S2]{BD01} for a detailed discussion.

\subsubsection{} We let $X$ be a connected, smooth algebraic curve over $k$.

\subsubsection{} Let $\Ran(X)$ denote the Ran space associated to $X$, regarded as a prestack (in fact, a presheaf). For an affine test scheme $S$ over $k$, an element of $\op{Maps}(S,\Ran(X))$ is by definition a finite subset $x^I = \{x^{(1)},\cdots,x^{(|I|)}\}$ of $\op{Maps}(S, X)$.

\smallskip

The prestack $\Ran(X)$ has an explicit presentation as a colimit of schemes. Let $\mathbf{fSet}^{\tn{surj}}$ denote the category of finite nonempty sets with surjections as morphisms. Then we have an equivalence:
$$
\Ran(X) \xrightarrow{\sim} \underset{I\in\mathbf{fSet}^{\op{surj}}}{\op{colim}}\, X^I,
$$
where for each $I_1 \twoheadrightarrow I_2$, the corresponding map $X^{I_2} \rightarrow X^{I_1}$ is the diagonal embedding. We refer the reader to \cite[\S1]{Ga13} for basic properties of the Ran space.

\subsubsection{}
For a finite nonempty set $I$, we let $\Ran(X)^{\times I}_{\tn{disj}}$ denote the open locus in $\Ran(X)^{\times I}$ where the sets of points associated to distinct elements $i_1\neq i_2\in I$ are pairwise disjoint.

\smallskip

A prestack $\cal Y$ over $\Ran(X)$ is a \emph{factorization prestack} if its pullback $\sqcup^*\cal Y$ along the map of taking disjoint union:
$$
\sqcup : \Ran(X)^{\times I}_{\op{disj}}\rightarrow\Ran(X)
$$
comes equipped with an identification with the restriction $\cal Y^{\times I}\big|_{\Ran(X)^{\times I}_{\op{disj}}}$ for each $I$. This identification is required to satisfy the obvious compatibility condition for compositions along surjections of finite nonempty sets $I_1\twoheadrightarrow I_2$.

\subsubsection{}
Let $\cal Y$ be a factorization prestack over $\Ran(X)$. A \emph{factorization line bundle} on $\cal Y$ is a line bundle $\cal L$ together with an isomorphism
\begin{equation}
\label{eq-pic-fact-isom}
\sqcup^*\cal L\xrightarrow{\sim} \cal L^{\boxtimes I}\big|_{\cal Y^{\times I}|_{\Ran(X)^{\times I}_{\op{disj}}}}
\end{equation}
over the factorization isomorphism $\sqcup^*\cal Y\xrightarrow{\sim} \cal Y^{\times I}\big|_{\Ran(X)^{\times I}_{\op{disj}}}$, satisfying compatibility for compositions. In fact, it suffices to specify isomorphisms \eqref{eq-pic-fact-isom} for $|I| = 2$, and check the compatilibity conditions for $|I| \le 3$.

\subsubsection{} Let $G$ be a connected, reductive group over $k$. We write $G_{\op{der}}$ for the derived subgroup of $G$, and $\widetilde G_{\op{der}}$ for its universal cover.

\smallskip

When we have fixed a maximal torus $T\subset G$, the notations $T_{\tn{der}}$ and $\widetilde T_{\tn{der}}$ will be used to denote the induced maximal tori in $G_{\tn{der}}$ and $\widetilde G_{\tn{der}}$.

\subsubsection{} We write $\op{Gr}_G$ for the Beilinson-Drinfeld affine Grassmannian associated to $G$. For a test affine scheme $S$, the set $\op{Maps}(S,\op{Gr}_G)$ consists of triples $(\{x^I\}, \cal P_G, \alpha)$, where:
\begin{enumerate}[(a)]
	\item $x^I$ is a finite subset of $\op{Maps}(S,X)$;
	\item $\cal P_G$ is a(n \'etale-locally trivial) $G$-bundle over $S\times X$;
	\item $\alpha$ is a trivialization of $\cal P_G$ over $S\times X-\bigcup_{i\in I}\Gamma_{x^{(i)}}$, where $\Gamma_{x^{(i)}}$ denotes the graph of $x^{(i)}$.
\end{enumerate}
The morphism $\op{Gr}_G\rightarrow\Ran(X)$ is ind-schematic and of ind-finite type, and realizes $\op{Gr}_G$ as a factorization prestack over $\Ran(X)$. The base change of $\op{Gr}_{G}$ along $X^I\rightarrow\Ran(X)$ will be denoted by $\op{Gr}_{G,X^I}$. We refer the reader to \cite{Zh16} for properties of $\tn{Gr}_G$.

\subsubsection{} We let $\cal L^+G$ (resp.~$\cal LG$) denote the arc (resp.~loop) group, regarded as factorization group prestacks over $\Ran(X)$. For a test affine scheme $S$, a lift of $x^I : S\rightarrow \Ran(X)$ to $\cal L^+G$ (resp.~$\cal LG$) is given by a map from the formal completion $D_{x^I}$ (resp.~punctured formal completion $\mathring D_{x^I} := D_{x^I}\backslash\bigcup_{i\in I}\Gamma_{x^{(i)}}$) of $\bigcup_{i\in I}\Gamma_{x^{(i)}}$ inside $S\times X$ to $G$.

\smallskip

Furthermore, the projection $\cal L^+G\rightarrow\Ran(X)$ is schematic (but not of finite type) and $\cal LG\rightarrow\Ran(X)$ is ind-schematic. The affine Grassmannian $\op{Gr}_G$ can be expressed as the quotient $\cal LG/\cal L^+G$ of \'etale sheaves.

\subsubsection{} For a closed point $x\in X$, we denote by $\cal O_x$ the \emph{completed} local ring at $x$ and $\cal K_x$ its localization at a uniformizer. The fibers of the above prestacks at a closed point $x\in X$ will be denoted by $\op{Gr}_{G,x}$, $\cal L_xG$, and $\cal L_x^+G$. Thus $\cal L_xG(k) \cong G(\cal K_x)$ and $\cal L_x^+G(k) \cong G(\cal O_x)$.

\subsection*{Acknowledgements} We thank D.~Gaitsgory for suggesting this problem to us, and for many insights that played a substantial role in its solution. We also benefited from discussions with Justin Campbell, Elden Elmanto, Quoc P.~Ho, and Xinwen Zhu.

\medskip

\section{Factorization line bundles on $\op{Gr}_T$}
\label{sec-torus}

In this section, we prove that factorization line bundles on various versions of $\op{Gr}_T$ (e.g.,~combinatorial, rational) are all classified by $\theta$-data.

\subsection{The many faces of $\op{Gr}_T$}

\subsubsection{}
Suppose $T$ is a torus over $k$. Let $\Lambda_T$ denote its co-character lattice. We will first introduce a few variants of the affine Grassmannian $\tn{Gr}_T$. They are summarized in the following commutative diagram:
\begin{equation}
\label{eq-many-Gr}
\xymatrix@C=1.5em@R=1.5em{
	\op{Gr}_{T,\op{comb}} \ar[r] & \op{Gr}_T \ar[r]\ar[d] & \op{Div}(X)\underset{\mathbb Z}{\otimes}\Lambda_T \ar[d] \\
	& \op{Gr}_{T,\op{lax}} \ar[r] & \op{Gr}_{T,\op{rat}}
}
\end{equation}

\subsubsection{The combinatorial variant} Consider an index category whose objects are pairs $(I,\lambda^{(I)})$, where $I$ is a finite set, and $\lambda^{(I)}$ is an $I$-family of elements in $\Lambda_T$ (its element corresponding to $i\in I$ is denoted by $\lambda^{(i)}$). A morphism $(I,\lambda^{(I)})\rightarrow (J,\lambda^{(J)})$ in this category consists of a \emph{surjective} map $\varphi: I\twoheadrightarrow J$ such that $\lambda^{(j)}=\sum_{i\in\varphi^{-1}(j)}\lambda^{(i)}$ for all $j\in J$. We set:
$$
\op{Gr}_{T,\op{comb}} := \underset{(I,\lambda^{(I)})}{\op{colim}}\, X^I.
$$

$\op{Gr}_{T,\op{comb}}$ is a factorization prestack over $\Ran(X)$. Furthermore, we have a canonical map $\op{Gr}_{T,\op{comb}}\rightarrow\op{Gr}_T$ sending an $S$-point $x^I : S\rightarrow X^I$ corresponding to $(I,\lambda^{(I)})$ to the triple $(\{x^{(i)}\}, \bigotimes_{i\in I}\cal O(\lambda^{(i)}\Gamma_{x^i}),\alpha)$ where $\alpha$ is the tautological trivialization.

\subsubsection{The lax variant} We let $\op{Gr}_{T,\op{lax}}$ denote the lax prestack\footnote{See \cite[\S2]{Ga15} for an introduction to lax prestacks.} whose value at $S$ is the \emph{category} whose objects are triples $(x^I, \cal P_T, \alpha)$ as in $\op{Gr}_T(S)$, but there is a morphism:
$$
(x^I, \cal P_T, \alpha) \rightarrow (x^J, \cal P'_T, \alpha'),
$$
whenever $x^I\subset x^J$, $\cal P_T\xrightarrow{\sim}\cal P'_T$, and the trivialization $\alpha$ restricts to $\alpha'$ over the complement of $\bigcup_{j\in J}\Gamma_{x^{(j)}}$. Such a morphism is non-invertible when $x^I\subset x^J$ is a proper inclusion.

$\op{Gr}_{T,\op{lax}}$ is a factorization lax prestack over the lax version of the Ran space $\op{Ran}(X)_{\op{lax}}$. Furthermore, we have a canonical map $\op{Gr}_T\rightarrow\op{Gr}_{T,\op{lax}}$ sending $(x^I, \cal P_T, \alpha)$ to the very same object.

\subsubsection{The rational variant} We define $\op{Gr}_{T,\op{rat}}$ as a prestack whose value at $S$ is the groupoid of $T$-bundles $\cal P_T$ over $S\times X$ equipped with a \emph{rational trivialization}, i.e., for some open $U\subset S\times X$ which is schematically dense after arbitrary base change $S'\rightarrow S$, the $T$-bundle $\cal P_T$ admits a trivialization over $U$; we regard two rational trivializations as equivalent if they agree on the overlaps.

Even though $\op{Gr}_{T,\op{rat}}$ does not live over any version of the Ran space, one can still make sense of factorization line bundles (or any other gadget) over $\op{Gr}_{T,\op{rat}}$. Namely, it is a line bundle $\cal L$ over $\op{Gr}_{T,\op{rat}}$ together with isomorphisms:
$$
c_{\cal P_T^{(1)},\cal P_T^{(2)}} : \cal L\big|_{\cal P_T} \xrightarrow{\sim} \cal L\big|_{\cal P_T^{(1)}} \otimes \cal L\big|_{\cal P_T^{(2)}},
$$
whenever $\cal P_T^{(1)}$ (resp.~$\cal P_T^{(2)}$) admits a trivialization over $U^{(1)}$ (resp.~$U^{(2)}$) such that the complements of $U^{(1)}$ and $U^{(2)}$ are disjoint, and $\cal P_T$ is the gluing of $\cal P_T^{(1)}\big|_{U^{(2)}}$ and $\cal P_T^{(2)}\big|_{U^{(1)}}$ along $U^{(1)}\cap U^{(2)}$, where they are both trivialized. The isomorphisms $c_{\cal P_T^{(1)},\cal P_T^{(2)}}$ are required to satisfy the obvious compatibility conditions in the presence of three $T$-bundles.

\begin{rem}
The objects $\op{Gr}_{T,\op{lax}}$ and $\op{Gr}_{T,\op{rat}}$ have analogues for a general group $G$, but we will not use them in this paper.
\end{rem}

\subsubsection{Colored divisors}
Recall the prestack $\op{Div}(X)$ whose value at $S$ is the abelian group of Cartier divisors of $S\times X$ relative to $S$. We take $\op{Div}(X)\underset{\mathbb Z}{\otimes}\Lambda_T$ as its extension of scalars to $\Lambda_T$. There is a morphism $\op{Div}(X)\rightarrow\op{Gr}_{\mathbb G_m,\op{rat}}$ defined by associating to a Cartier divisor $D$ the line bundle $\cal O_{S\times X}(D)$. It extends to a morphism $\op{Div}(X)\underset{\mathbb Z}{\otimes}\Lambda_T\rightarrow\op{Gr}_{T,\op{rat}}$.

As in the previous case, we make sense of factorization line bundles over $\op{Div}(X)\underset{\mathbb Z}{\otimes}\Lambda_T$ as follows. It is a line bundle $\cal L$ together with isomorphisms:
$$
c_{D_1,D_2} : \cal L\big|_{D_1+D_2} \xrightarrow{\sim} \cal L\big|_{D_1}\otimes\cal L\big|_{D_2},
$$
whenever the support of $D_1$ and $D_2$ are disjoint. The isomorphisms $c_{D_1,D_2}$ are required to satisfy the obvious compatibility conditions for three divisors.

\subsection{Classification statements}

\subsubsection{$\theta$-data}
\label{sec-theta}
We recall the notion of $\theta$-data for a lattice $\Lambda$ due to Beilinson-Drinfeld \cite[\S3.10.3]{BD04}. The Picard groupoid $\theta(\Lambda)$ consists of triples $(q, \cal L^{(\lambda)}, c_{\lambda,\mu})$ where:
\begin{enumerate}[(a)]
	\item $q\in Q(\Lambda,\mathbb Z)$ is an integral valued quadratic form on $\Lambda$; we use $\kappa$ to denote its symmetric bilinear form, defined by the formula: $\kappa(\lambda,\mu):= q(\lambda+\mu) - q(\lambda) - q(\mu)$; 
	\item $\cal L^{(\lambda)}$ is a system of line bundles on $X$ parametrized by $\lambda\in \Lambda$, and
	\item $c_{\lambda,\mu}$ are isomorphisms:
	\begin{equation}
	\label{eq-theta-data-isom}
	c_{\lambda,\mu} : \cal L^{(\lambda)} \otimes \cal L^{(\mu)} \xrightarrow{\sim} \cal L^{(\lambda+\mu)}\otimes\omega_X^{\kappa(\lambda,\mu)},
	\end{equation}
	which are associative, and satisfy a \emph{$\kappa$-twisted} commutativity condition, i.e.
	\begin{equation}
	\label{eq-pic-twisted-commutative}
	c_{\lambda,\mu}(a\otimes b)=(-1)^{\kappa(\lambda,\mu)}\cdot c_{\mu,\lambda}(b\otimes a).
	\end{equation}
\end{enumerate}

\begin{rem}
The authors of \cite{BD04} work in the setting of $\mathbb Z/2\mathbb Z$-graded line bundles, so what we call $\theta$-data corresponds to what they call \emph{even} $\theta$-data.
\end{rem}

\subsubsection{Shifted $\theta$-data}
\label{sec-shifted-theta}
For later purposes, we also introduce a Picard groupoid $\theta^+(\Lambda)$ consisting of pairs $(q,\cal L^{(\lambda)}, c^+_{\lambda,\mu})$, where we replace \eqref{eq-theta-data-isom} by isomorphisms $c_{\lambda,\mu}^+ : \cal L^{(\lambda)}\otimes\cal L^{(\mu)}\xrightarrow{\sim} \cal L^{(\lambda+\mu)}$ and also demand that they are associative and satisfy the $\kappa$-twisted commutativity condition. Clearly, we have an equivalence:
$$
\theta(\Lambda) \xrightarrow{\sim} \theta^+(\Lambda),\quad (q,\cal L^{(\lambda)})\leadsto (q, \cal L^{(\lambda)}\otimes\omega_X^{q(\lambda)}).
$$

\begin{lem}
\label{lem-classify-Gr-comb}
There is a canonical equivalence of Picard groupoids $\mathbf{Pic}^{\op{fact}}(\op{Gr}_{T,\op{comb}}) \xrightarrow{\sim} \theta(\Lambda_T)$.
\end{lem}
\begin{proof}
Given a factorization line bundle over $\op{Gr}_{T,\op{comb}}$, we denote its pullback along the inclusion $X \rightarrow \op{Gr}_{T,\op{comb}}$ corresponding to $(\{1\}, \lambda)$ by $\cal L^{(\lambda)}$, and its pullback along $X^2\rightarrow\op{Gr}_{T,\op{comb}}$ corresponding to $(\{1,2\},(\lambda,\mu))$ by $\cal L^{(\lambda,\mu)}$. The factorization isomorphism shows that there is an isomorphism $\cal L^{(\lambda)}\boxtimes\cal L^{(\mu)}\big|_{x^2-\Delta}\xrightarrow{\sim}\cal L^{(\lambda,\mu)}$. It extends to an isomorphism
\begin{equation}
\label{eq-fact-isom-ext}
\cal L^{(\lambda)}\boxtimes\cal L^{(\mu)}\xrightarrow{\sim} \cal L^{(\lambda,\mu)}\otimes\cal O_{X^2}(-\kappa(\lambda,\mu)\Delta),
\end{equation}
for some uniquely determined integer $\kappa(\lambda,\mu)$; its dependency on $\lambda,\mu$ is bilinear, by considering $\cal L^{(\lambda,\mu,\nu)}$ for a triple $(\{1,2,3\},(\lambda,\mu,\nu))$, using the compatibility between factorization isomorphism and composition. Since $\cal L^{(\lambda,\mu)}$ restricts to $\cal L^{(\lambda+\mu)}$ along $\Delta\hookrightarrow X^2$, the isomorphism \eqref{eq-fact-isom-ext} restricts to a system of isomorphisms $c_{\lambda,\mu}$ as in \eqref{eq-theta-data-isom}.

Next, because the factorization isomorphisms are $\Sigma_2$-invariant, so are the isomorphisms \eqref{eq-fact-isom-ext}. In other words, we have a commutative diagram:
\begin{equation}
\label{eq-fact-compatible-with-swap}
\xymatrix{
	\cal L^{(\lambda)} \boxtimes \cal L^{(\mu)} \ar[r]^-{\sim}\ar[d]^{\cong} & \cal L^{(\lambda,\mu)} \otimes \cal O_{X^2}(-\kappa(\lambda,\mu)\Delta) \ar[d]^{\cong} \\
	\sigma^*(\cal L^{(\mu)} \boxtimes \cal L^{(\lambda)}) \ar[r]^-{\sim} & \sigma^*\cal L^{(\mu,\lambda)} \otimes \sigma^*\cal O_{X^2}(-\kappa(\mu,\lambda)\Delta),
}
\end{equation}
where $\sigma$ is the isomorphism $X^{(\lambda,\mu)}\xrightarrow{\sim} X^{(\mu,\lambda)}$. One deduces from this fact that $\kappa$ is also symmetric. Restricting \eqref{eq-fact-compatible-with-swap} to the diagonal, we obtain a commutative diagram:
$$
\xymatrix{
	\cal L^{(\lambda)} \otimes \cal L^{(\mu)} \ar[r]^-{c_{\lambda,\mu}}\ar[d]^{\cong} & \cal L^{(\lambda + \mu)} \otimes \omega_X^{\kappa(\lambda,\mu)} \ar[d]^{(-1)^{\kappa(\lambda,\mu)}} \\
	\cal L^{(\mu)} \otimes \cal L^{(\lambda)} \ar[r]^-{c_{\mu,\lambda}} & \cal L^{(\mu + \lambda)} \otimes \omega_X^{\kappa(\mu,\lambda)}
}
$$
where the multiplication by $(-1)^{\kappa(\lambda,\mu)}$ appears because the isomorphism $\cal O_{X^2}(-\Delta)\big|_{\Delta}\xrightarrow{\sim}\omega_X$ is only $\Sigma_2$-invariant \emph{up to a sign}. This commutative diagram expresses the identity \eqref{eq-pic-twisted-commutative}. Finally taking $\lambda = \mu$, we see that $(-1)^{\kappa(\lambda,\lambda)} = 1$, so $\kappa(\lambda, \lambda) = 2q(\lambda)$ for a uniquely determined integer $q(\lambda)$. Thus we have define an integral quadratic form $q$ on $\Lambda_T$.

The above procedure defines the functor $\mathbf{Pic}^{\op{fact}}(\op{Gr}_{T,\op{comb}})\rightarrow\Theta(\Lambda_T; \mathbf{Pic})$. Checking that it is an equivalence is straightforward.
\end{proof}

\subsubsection{} We can now state the main result of this section. By pulling back along the morphisms of \eqref{eq-many-Gr}, we obtain a diagram of Picard groupoids, where the leftmost equivalence comes from Lemma \ref{lem-classify-Gr-comb}:
\begin{equation}
\label{eq-many-Gr-pic}
\xymatrix@C=1.5em@R=1.5em{
	\theta(\Lambda_T) & \mathbf{Pic}^{\op{fact}}(\op{Gr}_{T,\op{comb}}) \ar[l]_-{\sim} & \mathbf{Pic}^{\op{fact}}(\op{Gr}_T) \ar[l] & \mathbf{Pic}^{\op{fact}}(\op{Div}(X)\underset{\mathbb Z}{\otimes}\Lambda_T) \ar[l]_-{(a)} \\
	& & \mathbf{Pic}^{\op{fact}}(\op{Gr}_{T,\op{lax}}) \ar[u]^{(c)} & \mathbf{Pic}^{\op{fact}}(\op{Gr}_{T,\op{rat}}) \ar[l]_-{(b)}\ar[u].
}
\end{equation}

\begin{prop}
\label{prop-classify-torus-case}
All morphisms in \eqref{eq-many-Gr-pic} are equivalences.
\end{prop}
\begin{proof}
We shall deduce from existing literature how each of the labeled maps is an equivalence:
\begin{enumerate}[(a)]
	\item By \cite[\S3.10.7, Proposition]{BD04}, the composition of the top row defines an equivalence: $\mathbf{Pic}^{\op{fact}}(\op{Div}(X)\underset{\mathbb Z}{\otimes}\Lambda_T)\xrightarrow{\sim}\theta(\Lambda_T)$. This shows that the map $(a)$ has a left inverse.
	\item By \cite[Proposition 5.2.2]{Ba12}, the map $\op{Gr}_{T,\op{lax}}\rightarrow\op{Gr}_{T,\op{rat}}$ induces an equivalence after fppf sheafification. Hence pulling back defines an equivalence $\mathbf{Pic}(\op{Gr}_{T,\op{rat}})\xrightarrow{\sim}\mathbf{Pic}(\op{Gr}_{T,\op{lax}})$. One immediately checks that the additional data defining factorization structures on both are also equivalent. Hence $(b)$ is an equivalence.
	\item By \cite[Theorem 4.3.9(2)]{Zh16}, pulling back along $\op{Gr}_{T}\rightarrow\op{Gr}_{T,\op{rat}}$ defines an equivalence on \emph{rigidified} line bundles\footnote{\cite[Theorem 4.3.9(2)]{Zh16} is not given a proof in \emph{loc.cit.}, and we refer the reader to \cite{Ta19} for a complete proof of the key Pic-contractibility statement involved.}. On the other hand, every factorization line bundle on $\op{Gr}_T$ pulls back to one along the unit section $\Ran(X)\rightarrow\op{Gr}_T$, which is canonically trivial by Lemma \ref{lem-classify-Gr-comb} (applied to the trivial group). Thus a factorization line bundle on $\op{Gr}_T$ descends to a line bundle on $\op{Gr}_{T,\op{rat}}$, and the result has a canonical factorization structure as well, so we have an equivalence $\mathbf{Pic}^{\op{fact}}(\op{Gr}_{T,\op{rat}})\xrightarrow{\sim} \mathbf{Pic}^{\op{fact}}(\op{Gr}_T)$. This shows that $(c)$ is an equivalence.
\end{enumerate}
The undecorated maps in \eqref{eq-many-Gr-pic} are now equivalences by the 2-out-of-3 property.
\end{proof}

\begin{rem}
When $X$ is proper, \cite[Theorem 2.3.3]{Ca17} shows that the map $\op{Div}(X)\underset{\mathbb Z}{\otimes}\Lambda_T\rightarrow\op{Gr}_{T,\op{rat}}$ is an isomorphism of prestacks, which immediately implies that factorization line bundles on them are equivalent.
\end{rem}

\begin{rem}
We have the following equivalence for any smooth, fiberwise connected, affine group scheme $\mathbf G$ over $X$:
$$
\mathbf{Pic}^{\op{fact}}(\op{Gr}_{\mathbf G,\op{rat}})\xrightarrow{\sim}\mathbf{Pic}^{\op{fact}}(\op{Gr}_{\mathbf G,\op{lax}})\xrightarrow{\sim}\mathbf{Pic}^{\op{fact}}(\op{Gr}_{\mathbf G}).
$$
This is because the results \cite[Proposition 5.2.2]{Ba12} and \cite[Theorem 4.3.9(2)]{Zh16} both hold in this general context.
\end{rem}

\medskip

\section{Compatibility with the Brylinski-Deligne classification}
\label{sec-compatible}

In this section, we first summarize Brylinski-Deligne's classification of central extensions of $G$ by $\mathbf K_2$. Then we construct a functor from $\mathbf{Pic}^{\op{fact}}(\op{Gr}_G)$ to the same classification data and we prove that it is compatible with Gaitsgory's functor $\Phi_G$.

\subsection{Extensions by $\mathbf K_2$}
\label{sec-brde}

\subsubsection{} This subsection serves as a summary of the main result of \cite{BD01}. Let $G$ be a connected, reductive group over $k$. Fix a maximal torus $T\subset G$. We recall the notations $\theta(\Lambda_T)$ and $\theta^+(\Lambda_T)$ for the $\theta$-data associated to $\Lambda_T$ (see \S\ref{sec-theta}-\ref{sec-shifted-theta}).

\subsubsection{} We let $\mathbf K_2$ denote the Zariski sheafification of the presheaf on $\mathbf{Sch}^{\op{aff}}_{/X}$ that sends any $S\rightarrow X$ to $K_2(S)$. For a connected, reductive group $G$, we let $\mathbf{CExt}(G, \mathbf K_2)$ denote the Picard groupoid of central extensions
\begin{equation}
\label{eq-brde-data}
1 \rightarrow \mathbf K_2 \rightarrow E \rightarrow G \rightarrow 1,
\end{equation}
in the category of Zariski sheaves of groups on $\mathbf{Sch}_{/X}^{\op{aff}}$. This is Picard groupoid of \emph{Brylinski--Deligne data}.

\subsubsection{}
\label{sec-brde-class-torus} We will first define a functor
\begin{equation}
\label{eq-brde-class-torus}
\mathbf{CExt}(T, \mathbf K_2) \rightarrow \theta^+(\Lambda_T).
\end{equation}
Indeed, given a central extension $E$ of $T$, we construct a triple $(q, \cal L^{(\lambda)}, c_{\lambda,\mu}^+) \in \theta^+(\Lambda_T)$ from the following procedure:

\begin{enumerate}[(a)]
	\item The commutator in $E$ defines a map $\op{comm} : T\underset{\mathbb Z}{\otimes} T\rightarrow\mathbf K_2$ of Zariski sheaves on $\mathbf{Sch}^{\op{aff}}_{/X}$. For any $\lambda,\mu\in\Lambda_{T}$, the composition: $\mathbb G_m\underset{\mathbb Z}{\otimes}\mathbb G_m\xrightarrow{\lambda\otimes\mu} T\underset{\mathbb Z}{\otimes}T\rightarrow\mathbf K_2$ is some integral multiple of the universal symbol $\{-,-\}$ (c.f. \S3.8 of \emph{loc.cit.}). We call this integer $\kappa(\lambda,\mu)$. One then checks that $\kappa(-,-)$ is the bilinear form associated to some quadratic form $q$.
	
	\item Consider the projection $p : \mathbb G_m\times X\rightarrow X$. Using the vanishing result $\op R^1p_*\mathbf K_2=0$ of Sherman (c.f.~\S3.1 of \emph{loc.cit.}), we find an exact sequence of Zariski sheaves on $X$:
	$$
	1\rightarrow p_*\mathbf K_2 \rightarrow p_*E\rightarrow p_*T\rightarrow 1.
	$$
	Pushing out along the symbol map $p_*\mathbf K_2\rightarrow\mathbf K_1\cong\cal O_X^{\times}$, we obtain a multiplicative $\cal O_{X}^{\times}$-torsor over $p_*T$. The line bundle $\cal L^{(\lambda)}$ then arises as the fiber of the section of $p_*T$ defined by $\lambda\in\Lambda_{T}$.
	
	\item Note that the aforementioned multiplicative $\cal O_{X}$-torsor over $p_*T$ equips the system $\{\cal L^{(\lambda)}\}$ with the multiplicative structure $c_{\lambda,\mu}^+$. Its failure of commutativity is measured by $\kappa$, as desired.
\end{enumerate}

\subsubsection{}
It is proved in \emph{loc.cit.} that \eqref{eq-brde-class-torus} is an equivalence of Picard groupoids. We record here the \emph{un}shifted version of this equivalence:
\begin{equation}
\label{eq-brde-class-torus-unshifted}
\mathbf{CExt}(T,\mathbf K_2) \xrightarrow{\sim} \theta(\Lambda_{T}),
\end{equation}
i.e., it is the composition of \eqref{eq-brde-class-torus} with the equivalence of Picard groupoids $\theta^+(\Lambda_{T})\xrightarrow{\sim}\theta(\Lambda_{T})$ sending $\cal L^{(\lambda)}$ to $\cal L^{(\lambda)}\otimes\omega_X^{-q(\lambda)}$.

\subsubsection{}
\label{sec-brde-to-form}
We now turn to the general case. Note that there is always a functor:
\begin{equation}
\label{eq-brde-to-form}
\mathbf{CExt}(G, \mathbf K_2)\xrightarrow{\op{res}} \mathbf{CExt}(T, \mathbf K_2) \xrightarrow{\sim} \theta(\Lambda_{T}) \rightarrow Q(\Lambda_{T}, \mathbb Z),
\end{equation}
whose image lands in the $W$-invariant part of $Q(\Lambda_{T}, \mathbb Z)$. Thus, we may speak of \emph{the} quadratic form $q$ associated to an extension \eqref{eq-brde-data}.

\subsubsection{}
Suppose $G$ is semisimple and simply connected. Then Theorem 4.7 of \emph{loc.cit.} asserts that \eqref{eq-brde-to-form} defines an equivalence: $\mathbf{CExt}(G,\mathbf K_2)\xrightarrow{\sim} Q(\Lambda_{T},\mathbb Z)^{W}$. Thus for a semisimple, simply connected group $G$, there is a map which associates theta data to a $W$-invariant quadratic form:
\begin{equation}
\label{eq-ss-theta}
Q(\Lambda_{T}, \mathbb Z)^{W} \rightarrow \theta(\Lambda_{T}).
\end{equation}

\subsubsection{}
Let $\widetilde G_{\op{der}}$ be the simply connected cover of $G_{\op{der}}$. It contains a maximal torus $\widetilde T_{\op{der}}$ which is the preimage of $T_{\op{der}}$. We now let $\theta_G(\Lambda_{T})$ denote the Picard groupoid classifying:
\begin{enumerate}[(a)]
	\item a theta datum $(q,\cal L^{(\lambda)},c_{\lambda,\mu})$ for $\Lambda_{T}$, where $q$ is Weyl-invariant.
	\item an isomorphism $\varphi$ between the following theta data for $\Lambda_{\widetilde T_{\op{der}}}$:
	\begin{itemize}
		\item the restriction of $(q,\cal L^{(\lambda)},c_{\lambda,\mu})$ to $\Lambda_{\widetilde T_{\op{der}}}$;
		\item the theta data associated to $q\big|_{\Lambda_{\widetilde T_{\op{der}}}}$ via \eqref{eq-ss-theta}.
	\end{itemize}
\end{enumerate}

\noindent
In other words, $\varphi$ consists of isomorphisms between line bundles, preserving their ($\omega$-twisted) multiplicative structure. We shall call $\theta_G(\Lambda_{T})$ the Picard groupoid of \emph{enhanced} theta data. By definition, we have a functor:
\begin{equation}
\label{eq-brde-class}
\Phi_{\op{BD}} : \mathbf{CExt}(G, \mathbf K_2) \rightarrow \theta_G(\Lambda_{T}),
\end{equation}
obtained by restrictions to $T$ and $\widetilde T_{\op{der}}$. The main theorem of \cite{BD01} is that \eqref{eq-brde-class} is an equivalence of Picard groupoids, i.e., central extensions of $G$ by $\mathbf K_2$ are classified by enhanced theta data.

\subsection{Gaitsgory's functor $\Phi_G$}

\subsubsection{}
\label{sec-ga-functor} Under the condition that the characteristic of $k$ does not divide the integer $N_G$, Gaitsgory \cite{Ga18} constructed a functor:
\begin{equation}
\label{eq-ga-functor-text}
\Phi_G : \mathbf{CExt}(G, \mathbf K_2) \rightarrow \mathbf{Pic}^{\op{fact}}(\op{Gr}_G).
\end{equation}
Only two features of $\Phi_G$ will be used in proving its compatibility with the Brylinski-Deligne classification. We first cast them in informal language:
\begin{enumerate}[(a)]
	\item Given a central extension \eqref{eq-brde-data}, its image under $\Phi_G$ is a line bundle $\cal L$ over $\op{Gr}_G$ with additional factorization data; for a \emph{regular} affine scheme $S\rightarrow\op{Gr}_G$, we need the restriction $\cal L\big|_S$ to be given by ``taking the residue'' along $S\times X\rightarrow S$.
	\item Suppose $G=T$ is a torus; we need the functor $\Phi_{T}$ to factor through the Picard groupoid of \emph{multiplicative} factorization line bundles on $\cal LT$, and for a closed point $x\in X$, we need the multiplicative structure on $\cal L_xT$ to be given by the ``tautological'' one.
\end{enumerate}
\noindent
We will make precise what features (a) and (b) mean in the rest of this subsection, and explain how they can be deduced from \emph{loc.cit}.

\subsubsection{}
\label{sec-taking-residue} Let $S$ be a \emph{regular} affine scheme over $k$ and $\pi : \fr X\rightarrow S$ be a smooth relative curve, whose fibers are geometrically connected. Furthermore, suppose we have a finite set $\{x^I\}$ of sections $x^{(i)} : S\rightarrow\fr X$. Let $\Gamma_{x^I}$ denote the (scheme-theoretic) union of their images, and $U_{x^I}:=\fr X-\Gamma_{x^I}$ be its complement.

We will construct a functor, referred to hereafter as \emph{taking the residue} along $\pi$:
\begin{equation}
\label{eq-taking-the-residue}
\left\{\txt{$\mathbf K_2$-gerbes $\cal G$ on $\fr X$ with \\ neutralization $\gamma$ over $U_{x^I}$}\right\} \rightarrow \mathbf{Pic}(S).
\end{equation}
Indeed, the datum $(\cal G,\gamma)$ is equivalent to a section of $\iota^!\mathbf K_2[2]$ over $\fr X$, where $\iota : \Gamma_{x^I}\hookrightarrow\fr X$ is the closed immersion. On the other hand, the Gersten resolution of $\mathbf K_2$ on $\fr X$ shows that $\iota^!\mathbf K_2[2]$ is quasi-isomorphic to the complex concentrated in degrees $[-1,0]$:
\begin{equation}
\label{eq-gersten}
\bigoplus_{i\in I}(\iota_{\eta^{(i)}})_* K_1(\eta) \rightarrow \bigoplus_{\substack{\op{codim}(\nu)=1 \\ \text{in }\Gamma_{x^I}}} (\iota_{\nu})_*\mathbb Z
\end{equation}
where $\iota_{\eta^{(i)}}$ (resp.~$\iota_{\nu}$) denotes the inclusion of the generic point of the $i$th section (resp.~codimension-one point $\nu$ of $\Gamma_{x^I}$). On the other hand, $\mathbf K_1[1]$ over $S$ is quasi-isomorphic to:
$$
(\iota_{\eta})_*K_1(\eta) \rightarrow \bigoplus_{\substack{\op{codim}(\nu)=1 \\ \text{in } S}} (\iota_{\nu})_*\mathbb Z.
$$
Thus the image of \eqref{eq-gersten} under $\pi$ maps to $\mathbf K_1[1]$ via summation. Hence a section of $\iota^!\mathbf K_2[2]$ over $\fr X$ gives rise to a section of $\mathbf K_1[1]\cong\cal O_S^{\times}[1]$, i.e., a line bundle on $S$.

\subsubsection{}
\label{sec-smooth-test} Given an extension $E$ \eqref{eq-brde-data} and a map $S\rightarrow\op{Gr}_G$ specified by the triple $(\{x^I\}, \cal P_G, \alpha)$ where $\cal P_G$ is \emph{Zariski} locally trivial, we obtain a (Zariski) $\mathbf K_2$-gerbe $\cal G$ over $S\times X$, which classifies an $E$-torsor $\cal P_E$ equipped with an identification of its induced $G$-torsor $(\cal P_E)_G\xrightarrow{\sim}\cal P_G$. The trivialization $\alpha$ gives rise to a neutralization $\gamma$ of $\cal G$ over $U_{x^I}$.

Suppose $S$ is regular, then $(\cal G,\gamma)$ produces a line bundle on $S$ by taking the residue \eqref{eq-taking-the-residue} along $\pi : S\times X\rightarrow S$. This process also applies when $\cal P_G$ is only \'etale locally trivial, since \'etale locally on $S$ the bundle $\cal P_G$ becomes Zariski locally trivial (see \cite{DS95}). The fact that $\Phi_G(E)\big|_S$ naturally agrees with this line bundle is the content of \cite[\S2.3]{Ga18}; this is what we meant in part (a) of \S\ref{sec-ga-functor}.

\subsubsection{} Recall that a \emph{multiplicative} line bundle $\cal L$ on $\cal LG$ amounts to the additional isomorphism:
\begin{equation}
\label{eq-mult}
\op{mult}^*\cal L\xrightarrow{\sim} \cal L\boxtimes\cal L
\end{equation}
over $\cal LG\underset{\op{Ran}(X)}{\times}\cal LG$ that satisfies the cocycle condition on the triple product. If $\cal L$ is a factorization line bundle, then being multiplicative amounts to an isomorphism \eqref{eq-mult} that is compatible with the factorization structures on both sides.

We let $\mathbf{Pic}^{\op{fact},\times}(\cal LG)$ (resp.~$\mathbf{Pic}^{\op{fact},\times}_{/\cal L^+G}(\cal LG)$) denote the Picard groupoid of multiplicative factorization line bundles on $\cal LG$ (resp.~together with a trivialization \emph{as such} over $\cal L^+G$). Clearly, there is a descent functor:
$$
\mathbf{Pic}^{\op{fact},\times}_{/\cal L^+G}(\cal LG) \rightarrow \mathbf{Pic}^{\op{fact}}(\op{Gr}_G).
$$
We now state part (b) of \S\ref{sec-ga-functor} as a lemma:

\begin{lem}
\label{lem-mult-lift}
\begin{enumerate}[(a)]
	\item The functor $\Phi_T$ factors through $\mathbf{Pic}_{/\cal L^+T}^{\op{fact},\times}(\cal LT)$, i.e., $\Phi_T(E)$ has a canonical multiplicative structure over $\cal LT$, trivialized over $\cal L^+T$;
	\item Over a closed point $x\in X$, the restriction of the above multiplicative structure to the abstract group $T(\cal K_x)$\footnote{i.e., the group of $k$-points of $\cal L_xT$.} agrees with that on the $k^{\times}$-torsor coming from the push-out of
	\begin{equation}
	\label{eq-brde-at-loop}
	0 \rightarrow \mathbf K_2(\cal K_x) \rightarrow E(\cal K_x) \rightarrow T(\cal K_x) \rightarrow 0
	\end{equation}
	along the residue map $\mathbf K_2(\cal K_x)\rightarrow k^{\times}$. The same holds over any field extension $k\subset k'$.
\end{enumerate}
\end{lem}

\begin{rem}
Part (b) makes sense since $\Phi_T(E)\big|_t$ for $t\in T(\cal K_x)$ agrees with the $k^{\times}$-torsor induced from \eqref{eq-brde-at-loop}; this follows from the description of $\Phi_T(E)$ on regular test schemes (\S\ref{sec-smooth-test}).
\end{rem}

\begin{proof}[Proof of Lemma \ref{lem-mult-lift}]
Recall that $\cal L:=\Phi_T(E)$ is constructed as follows. The datum $E$ can be interpreted as a pointed morphism $e : X\times\op BT\rightarrow\op B^2\mathbf K_2$. Let $\mathbf K$ denote the full $K$-theory spectrum, regarded as a Zariski sheaf on $\mathbf{Sch}^{\op{aff}}$. Then $e$ lifts (non-uniquely) to some $\tilde e : X\times\op BT\rightarrow\mathbf K_{\ge 2}$ (\cite[\S5.3.1]{Ga18}). Hence the data $(\{x^I\},\cal P_T,\alpha)$ of an $S$-point of $\op{Gr}_T$ (where we may again assume $\cal P_T$ to be Zariski-locally trivial) give us a section of $\mathbf K_{\ge 2}$ over $S\times X$ with support on $\Gamma_{x^I}$. The line bundle $\cal L_{\tilde e}\big|_S$ is then constructed using the map:
\begin{equation}
\label{eq-map-from-full-sptr}
\tau^{\le 0}\pi_*\iota^!\mathbf K_{\ge 2} \rightarrow \cal O_S^{\times}[1]
\end{equation}
(c.f.~(3.2.2) of \emph{loc.cit.}). For two lifts $\tilde e$ and $\tilde e'$, we need to produce a canonical isomorphism $\cal L_{\tilde e}\xrightarrow{\sim} \cal L_{\tilde e'}$. This is done as follows:
\begin{enumerate}[(a)]
	\item for $S$ the spectrum of an Artinian $k$-algebra, \eqref{eq-map-from-full-sptr} factors through $\tau^{\le 0}\pi_*\iota^!\mathbf K_2$, so we obtain a \emph{canonical} isomorphism $\cal L_{\tilde e}\big|_S\xrightarrow{\sim} \cal L_{\tilde e'}\big|_S$;
	\item there exists an isomorphism $\cal L_{\tilde e}\xrightarrow{\sim}\cal L_{\tilde e'}$ which restricts to the one in (a) for any $S$ the spectrum of an Artinian $k$-algebra (\S5.3.4-6 of \emph{loc.cit.}).
\end{enumerate}

We now claim that $\cal L_{\tilde e}\big|_{\cal LT}$ acquires a canonical multiplicative structure. Indeed, $\tilde e$ induces a morphism $X\times T\rightarrow\mathbf K_{\ge 2}[-1]$ of group sheaves. Given $S$-points $t,t'$ of $\cal LT$ over the same point $x^I\in\op{Ran}(X)$, we may view them both as maps $\overset{\circ}{D}_{x^I} \rightarrow X\times T$. There is a canonical homotopy between $\tilde e(t) + \tilde e(t')$ and $\tilde e(tt')$ as maps $\overset{\circ}{D}_{x^I} \rightarrow\mathbf K_{\ge 2}[-1]$. Under the map $\mathbf K_{\ge 2}\big|_{\overset{\circ}{D}_{x^I}}[-1]\rightarrow\iota^!\mathbf K_{\ge 2}$ of sheaves over $D_{x^I}$, we obtain a canonical homotopy between the corresponding sections of $\iota^!\mathbf K_{\ge 2}$; it gives rise to the desired multiplicative structure $\cal L_{\tilde e}\big|_t\otimes\cal L_{\tilde e}\big|_{t'}\xrightarrow{\sim}\cal L_{\tilde e}\big|_{tt'}$ under \eqref{eq-map-from-full-sptr}.

It remains to check that for two lifts $\tilde e$ and $\tilde e'$, the canonical isomorphism $\cal L_{\tilde e}\xrightarrow{\sim}\cal L_{\tilde e'}$ is compatible with the multiplicative structures on both sides. This amounts to checking that the following diagram of line bundles over $\cal LT\underset{\op{Ran}(X)}{\times}\cal LT$ commutes:
$$
\xymatrix@C=1.5em@R=1.5em{
\op{mult}^*\cal L_{\tilde e} \ar[r]\ar[d] & \cal L_{\tilde e}\boxtimes\cal L_{\tilde e}\ar[d] \\
\op{mult}^*\cal L_{\tilde e'} \ar[r] & \cal L_{\tilde e'}\boxtimes\cal L_{\tilde e'}.
}
$$
It suffices to test the commutativity over $S$ the spectrum of an Artinian $k$-algebra. Note again that for such $S$, \eqref{eq-map-from-full-sptr} factors through $\tau^{\le 0}\pi_*\iota^!\mathbf K_2$, so the construction of the multiplicative structure does \emph{not} require a lift of $e$. Therefore, we have equipped $\cal L$ with a canonical multiplicative structure over $\cal LT$.

Part (b) of the lemma is immediate from the above construction, applied to $S=\Spec(k)$ (or $\Spec(k')$ for a field extension $k\subset k'$).
\end{proof}

\subsection{Compatibility: torus case}
\label{sec-comp-torus}

\subsubsection{} Fix a torus $T$. Recall the equivalence of Proposition \ref{prop-classify-torus-case}:
\begin{equation}
\label{eq-GrT-to-theta}
\mathbf{Pic}^{\op{fact}}(\op{Gr}_T)\xrightarrow{\sim}\theta(\Lambda_{T}).
\end{equation}
The goal of this subsection is to prove:

\begin{lem}
\label{lem-torus-compatible}
The following diagram of Picard groupoids commutes functorially in $T$:
\begin{equation}
\xymatrix@C=0em@R=1.5em{
	\mathbf{CExt}(T, \mathbf K_2) \ar[rr]^-{\Phi_T}\ar[dr]_-{\eqref{eq-brde-class-torus-unshifted}} & & \mathbf{Pic}^{\op{fact}}(\op{Gr}_T) \ar[dl]^-{\eqref{eq-GrT-to-theta}} \\
	& \theta(\Lambda_T)
}
\end{equation}
\end{lem}

\begin{rem}
Although Lemma \ref{lem-torus-compatible} appears as the special case of Proposition \ref{prop-full-compatible} for $G=T$, its proof contains most of the technical difficulties.
\end{rem}

\subsubsection{Notations}
Fix an object $E$ of $\mathbf{CExt}(T,\mathbf K_2)$. We denote its image in $\theta^+(\Lambda_T)$ under \eqref{eq-brde-class-torus} by $(q, \cal L^{(\lambda)}_+, c_{\mu,\nu}^+)$, and its image under $\Phi_T$ by $\cal L$. The image of $\cal L$ in $\theta(\Lambda_T)$ will be denoted by $(q', \cal L^{(\lambda)}, c_{\mu,\nu})$. We ought to show:
\begin{enumerate}[(a)]
	\item $q=q'$;
	\item there is a canonical system of isomorphisms:
	\begin{equation}
	\label{eq-isom-line-bundles}
	\cal L_+^{(\lambda)} \xrightarrow{\sim} \cal L^{(\lambda)}\otimes\omega_X^{q(\lambda)}
	\end{equation}
	which respects $c_{\mu,\nu}^+$ and $c_{\mu,\nu}$.
\end{enumerate}

\subsubsection{Quadratic forms} We first show $q=q'$ by checking that their bilinear forms $\kappa$ and $\kappa'$ agree. Fixing a closed point $x\in X$ and any co-character $\mu\in\Lambda_T$, we will show that $\kappa(-, \mu)$ and $\kappa'(-,\mu)$ define the same character $T(k')\rightarrow\mathbb G_m(k')$ for every field extension $k\subset k'$; this will imply that $\kappa=\kappa'$.\footnote{Indeed, for every $\lambda\in\Lambda_T$, suppose $z\leadsto z^{\kappa(\lambda,\mu)}$ and $z\leadsto z^{\kappa'(\lambda,\mu)}$ define the same map $\mathbb G_m(k')\rightarrow\mathbb G_m(k')$ for all field extension $k\subset k'$. By suitably choosing $k'$, we can ensure that $(k')^{\times}$ contains an element of infinite order. Thus $\kappa(\lambda,\mu)$ agrees with $\kappa'(\lambda,\mu)$.}

We now further fix a uniformizer of the completed local ring $t\in\cal O_x$. This provides an isomorphism $k[\![t]\!]\xrightarrow{\sim}\cal O_x$, so we regard $t^{\mu}$ as an element of $T(\cal K_x)$. Consider the central extension \eqref{eq-brde-at-loop} corresponding to $x\in X$. Pushing-out along the residue map $\mathbf K_2(\cal K_x)\rightarrow k^{\times}$, we obtain central extension:
$$
0 \rightarrow k^{\times} \rightarrow E' \rightarrow T(\cal K_x) \rightarrow 0.
$$
So the conjugation action of $T(\cal O_x)$ on the fiber of $E(\cal K_x)\rightarrow T(\cal K_x)$ at $t^{\mu}$ induces a map:
\begin{equation}
\label{eq-conj-action}
T(\cal O_x)\rightarrow k^{\times}.
\end{equation}
We will calculate this map (and its variant for a field extension $k\subset k'$) in two ways.

\smallskip

\textbf{Step 1.} We first show that the map \eqref{eq-conj-action} is given by the composition:
$$
T(\cal O_x)\xrightarrow{\op{ev}} T(k)\xrightarrow{\kappa(-,\mu)} k^{\times}.
$$
Indeed, recall from \S\ref{sec-brde-class-torus}(a) that the composition $\mathbb G_m\underset{\mathbb Z}{\otimes}\mathbb G_m\xrightarrow{\lambda\otimes\mu} T\underset{\mathbb Z}{\otimes}T\xrightarrow{\op{comm}}\mathbf K_2$ is the $\kappa(\lambda,\mu)$-multiple of the universal symbol. Thus the map:
$$
\mathbb G_m(\cal K_x)\underset{\mathbb Z}{\otimes}\mathbb G_m(\cal K_x) \xrightarrow{\lambda\otimes\mu}T(\cal K_x) \underset{\mathbb Z}{\otimes} T(\cal K_x) \xrightarrow{\op{comm}} \mathbf K_2(\cal K_x) \xrightarrow{\op{res}} k^{\times}
$$
is the $\kappa(\lambda,\mu)$-multiple of the Contou-Carr\`ere symbol $\{f,g\}:=(f^{\op{ord}(g)}/g^{\op{ord}(f)})(0)$. Hence the conjugation action of $f\in\mathbb G_m(\cal O_x)$ (through $\lambda$) on $E'$ is given by $e'\leadsto \{f, t\}^{\kappa(\lambda,\mu)}e'$. Note that $\{f,t\}=f(0)$, as required.

For a field extension $k\subset k'$, the above computation holds without modification.

\smallskip

\textbf{Step 2.} We now calculate the map \eqref{eq-conj-action} alternatively as follows. Recall the canonical multiplicative structure on $\cal L\big|_{\cal LT}$ from Lemma \ref{lem-mult-lift}. It induces a \emph{strong} $\cal L^+T$-equivariance structure on $\cal L$ (over $\op{Gr}_T$, c.f.~\cite[\S7.3.4]{GL16}) with respect to the trivial left $\cal L^+T$-action; in other words, the twisted product $\cal L\widetilde{\boxtimes}\cal L$ on the convolution Grassmannian $\widetilde{\op{Gr}}_{T,X^2}$ is identified with the pullback of $\cal L^{(2)}$ along the action map $\widetilde{\op{Gr}}_{T,X^2} \rightarrow \op{Gr}_{T,X^2}$, in a way that is compatible with the factorization structure of $\cal L$.

Furthermore, its value at $\op{Gr}_{T,x}^{\mu}$ is given by the conjugation action \eqref{eq-conj-action}. We claim now that the map \eqref{eq-conj-action} is given by
$$
T(\cal O_x)\xrightarrow{\op{ev}} T(k)\xrightarrow{\kappa'(-,\mu)} k^{\times}.
$$
Indeed, this follows from the fact that for a factorization line bundle $\cal L$ on $\op{Gr}_T$ with associated bilinear form $\kappa'$, every strong $\cal L^+T$-equivariance structure acts on $t^{\mu}\in\op{Gr}_{T,x}$ through the composition $\cal L^+T\xrightarrow{\op{ev}} T \xrightarrow{\kappa'(-,\mu)} \mathbb G_m$ (c.f.~\cite[\S7.4]{GL16}).

Again for a field extension $k\subset k'$, the above computation holds without modification. This finishes the proof that $\kappa=\kappa'$.

\subsubsection{Isomorphisms of line bundles} We now construct the isomorphisms \eqref{eq-isom-line-bundles}. The strategy is to first identify $\cal L^{(\lambda)}$ with the twist of $\cal L_+^{(\lambda)}$ by some power of the tangent sheaf $\cal T_X$, and then determine this power.

\smallskip

\textbf{Step 1.} Consider the diagonal embedding $\Delta: X\hookrightarrow X\times X$. Define $\cal G^{(\lambda)}$ as the $\mathbf K_2$-gerbe on $X\times X$ classifying a $\op{pr}_2^*E$-torsor $\cal P_E$, together with an isomorphism $(\cal P_E)_T\xrightarrow{\sim}\cal O(\lambda\Delta)$. Then $\cal G^{(\lambda)}$ comes equipped with a neutralization $\gamma$ over $X\times X-\Delta$. The line bundle $\cal L^{(\lambda)}$ arises from $(\cal G^{(\lambda)}, \gamma)$ by taking the residue along $\op{pr}_1$ (c.f.~\S\ref{sec-taking-residue}).

Let $X\times\mathbb A^1 \hookrightarrow \fr X \rightarrow\mathbb A^1$ be the deformation of the diagonal embedding to the normal cone, constructed as the blow-up of $X\times X\times\mathbb A^1$ along the diagonally embedded subscheme $X\times\{0\}$, where we then remove the strict transform of $X\times X\times\{0\}$. It has the following features:
\begin{enumerate}[(a)]
	\item $X\times\{t\}\hookrightarrow\fr X\big|_t$ identifies with $X\hookrightarrow X\times X$ for $t\neq 0$;
	\item $X\times\{0\}\hookrightarrow \fr X\big|_{0}$ identifies with the embedding of $X$ as the zero section inside the total space of the tangent sheaf $T_X$.
	\item there is a canonical map $\fr X\xrightarrow{\op{pr}_1,\op{pr}_2} X\times X$ which is identity for $t\neq 0$, and the canonical projection $T_X \xrightarrow{p,p} X\times X$ at $t=0$.
\end{enumerate}
Consider $\fr Z:=X\times\mathbb A^1$ as a divisor inside $\fr X$. We define $\widetilde{\cal G}^{(\lambda)}$ as the $\mathbf K_2$-gerbe classifying a $\op{pr}_2^*E$-torsor $\widetilde{\cal P}_E$ over $\fr X$, together with an isomorphism $(\widetilde{\cal P}_E)_T\xrightarrow{\sim} \cal O(\lambda\fr Z)$. Note that $\widetilde{\cal G}^{(\lambda)}$ is equipped with a neutralization over $\fr X-\fr Z$, so we may take the residue along $\op{pr}_1$ to obtain a line bundle $\widetilde{\cal L}^{(\lambda)}$ over $X\times\mathbb A^1$.

Tautologically, $\widetilde{\cal L}^{(\lambda)}\big|_{X\times\{t\}}$ identifies with $\cal L^{(\lambda)}$ for $t\neq 0$. On the other hand, every line bundle on $X\times\mathbb A^1$ canonically identifies with the pullback of a line bundle from $X$. Thus, we obtain an isomorphism $\widetilde{\cal L}^{(\lambda)}\big|_{X\times\{t\}} \xrightarrow{\sim} \widetilde{\cal L}^{(\lambda)}\big|_{X\times\{0\}}$. This shows that $\cal L^{(\lambda)}$ arises from the residue of $(\cal G^{(\lambda)}_{T_X}, \gamma_{T_X})$ along $p : T_X\rightarrow X$, where:
\begin{enumerate}[(a)]
	\item $\cal G_{T_X}^{(\lambda)}$ is the $\mathbf K_2$-gerbe on $T_X$ classifying a $p^*E$-torsor $\cal P_E$, together with an isomorphism $(\cal P_E)_T\xrightarrow{\sim} \cal O(\lambda\{0\})$, where $\{0\}$ denotes the zero section $X\hookrightarrow T_X$; and
	\item $\gamma_{T_X}$ is the tautological neutralization of $\cal G_{T_X}^{(\lambda)}$ over $T_X-\{0\}$.
\end{enumerate}

\smallskip

\textbf{Step 2.} In the above description, suppose we replaced $p: T_X\rightarrow X$ by the trivial line bundle $\mathbb A^1_X\rightarrow X$; then the line bundle arising from taking the residue of the analogously defined pair $(\cal G_{\mathbb A^1_X}^{(\lambda)}, \gamma_{\mathbb A^1_X})$ would identify with $\cal L_+^{(\lambda)}$. Indeed, this follows from comparing the construction of \S\ref{sec-taking-residue} with that of \S\ref{sec-brde-class-torus}(b).

We now explain an alternative way to arrive at $\cal L^{(\lambda)}$ via twisting the line bundle $\mathbb A^1_X\rightarrow X$ in the above construction. Consider the $\mathbb G_m$-action on $\mathbb A^1_X$ by scaling. The pair $(\cal G_{\mathbb A^1_X}^{(\lambda)}, \gamma_{\mathbb A^1_X})$ admits a $\mathbb G_m$-equivariance structure. Hence $L_+^{(\lambda)}$ (the total space of $\cal L_+^{(\lambda)}$) is equipped with a fiberwise $\mathbb G_m$-action. Since $\cal G_{T_X}^{(\lambda)}$ identifies with the twisted product $\cal G^0\widetilde{\boxtimes}\cal G_{\mathbb A^1_X}^{(\lambda)}$ on the total space $T_X^{\times}\overset{\mathbb G_m}{\times}\mathbb A^1_X$ (where $\cal G^0$ denotes the trivial gerbe), we find $L^{(\lambda)}\xrightarrow{\sim} T_X^{\times}\overset{\mathbb G_m}{\times}L_+^{(\lambda)}$. In other words, suppose the fiberwise $\mathbb G_m$-action on $L_+^{(\lambda)}$ is given by some character $q_1(\lambda)\in\mathbb Z$, then there is a canonical isomorphism:
\begin{equation}
\label{eq-torus-compatible-rev}
\cal L^{(\lambda)} \xrightarrow{\sim} \cal T_X^{q_1(\lambda)}\otimes\cal L_+^{(\lambda)}.
\end{equation}

\smallskip

\textbf{Step 3.} We now calculate the character $q_1(\lambda)$.\footnote{Caution: we do not yet know that $q_1(\lambda)$ depends quadratically on $\lambda$.} It suffices to do so at a closed point $x\in X$. The line $L_+^{(\lambda)}\big|_{x\in X}$ admits a simple description as follows (c.f.~\S\ref{sec-brde-class-torus}). Evaluating $E$ at $\mathbb G_{m,x}:=\Spec(k[t,t^{-1}])$, we obtain an exact sequence:
\begin{equation}
\label{eq-brde-at-gm}
0 \rightarrow \mathbf K_2(k[t,t^{-1}]) \rightarrow E(k[t,t^{-1}]) \rightarrow T(k_x[t,t^{-1}]) \rightarrow 0,
\end{equation}
and consequently a $\mathbf K_2(k[t,t^{-1}])$-torsor $E(z)$ at every point $z\in T(k[t,t^{-1}])$. The line $L_+^{(\lambda)}\big|_{x\in X}$ is the $k^{\times}$-torsor induced from $E(t^{\lambda})$ along the residue map $\mathbf K_2(k[t,t^{-1}])\rightarrow k^{\times}$.

To unburden the notation, we again use $L_+^{(\lambda)}$ to denote this line; the $\mathbb G_m(k)$-action on it also admits a simple description. Take $a\in \mathbb G_m(k)$, the action by $a^{q_1(\lambda)}$:
\begin{equation}
\label{eq-auto-mult}
\cdot a^{q_1(\lambda)} : L_+^{(\lambda)}\big|_{x\in X}\xrightarrow{\sim} L_+^{(\lambda)}\big|_{x\in X}
\end{equation}
is given as follows.
\begin{enumerate}[(a)]
	\item Consider the scaling map $k[t,t^{-1}]\rightarrow k[t,t^{-1}]$, $t\leadsto t\cdot a$. It induces a group automorphism $E(k[t,t^{-1}])\xrightarrow{a_*} E(k[t,t^{-1}])$, covering the analogously defined automorphism on $T(k[t,t^{-1}])$. In particular, we obtain a map $a_* : E(t^{\lambda}) \rightarrow E(t^{\lambda}a^{\lambda})$ (\emph{incompatible} with the $\mathbf K_2(k[t,t^{-1}])$-torsor structures.)
	
	After inducing to $k^{\times}$-torsors, we obtain a map \emph{compatible} with the $k^{\times}$-torsor structures:
	$$
	a_* : L_+^{(\lambda)} \rightarrow L_+(t^{\lambda}a^{\lambda}) := E(t^{\lambda}a^{\lambda})_{k^{\times}},
	$$
	since $a_* : \mathbf K_2(k[t,t^{-1}]) \rightarrow \mathbf K_2(k[t,t^{-1}])$ induces the identity on $k^{\times}$.
	
	\item On the other hand, every element in $T(k[t])$ admits a lift to $E(k[t])$, up to an element from $\mathbf K_2(k[t])$ (as follows from $\op R^1p_*\mathbf K_2=0$ for $p : \mathbb A^1_S\rightarrow S$, c.f.~\cite[\S3.1]{BD01}) Hence we have another map $E(t^{\lambda})\rightarrow E(t^{\lambda}a^{\lambda})$, defined as right-multiplying by \emph{any} lift of $a^{\lambda}\in T(k[t])$.
	
	Inducing along $\mathbf K_2(k[t,t^{-1}])\rightarrow k^{\times}$, we again obtain a map of $k^{\times}$-torsors:
	$$
	R_{a^{\lambda}} : L_+^{(\lambda)} \rightarrow L_+(t^{\lambda}a^{\lambda}).
	$$
	Note that this map is independent of the choice of the lift.
	
	\item The automorphism \eqref{eq-auto-mult} identifies with the composition $R_{a^{\lambda}}^{-1}\circ a_*$.
\end{enumerate}

\smallskip

\textbf{Step 4.} We shall now deduce two identities:
\begin{align}
	q_1(2\lambda) - \kappa(\lambda,\lambda) &= 2\cdot q_1(\lambda) \label{eq-quadratic-calc-1} \\
	4\cdot q_1(\lambda) &= q_1(2\lambda) \label{eq-quadratic-calc-2}
\end{align}
The combination of these identities will show that $q_1(\lambda)=\frac{1}{2}\kappa(\lambda,\lambda)=q(\lambda)$. Then the desired isomorphism follows from \eqref{eq-torus-compatible-rev}.

\begin{proof}[Proof of \eqref{eq-quadratic-calc-1}]
This follows from the mutiplicative structure on $E(k[t,t^{-1}])$. Indeed, consider the following commutative diagrams:
$$
\xymatrix@C=2em@R=1.5em{
L_+^{(2\lambda)} \ar[r]^{a_*}\ar[d]^{\cong} & L_+(t^{2\lambda}a^{2\lambda}) \ar[d]^{\cong} \\
L_+^{(\lambda)}\otimes L_+^{(\lambda)} \ar[r]^-{a_*\otimes a_*} & L_+(t^{\lambda}a^{\lambda})\otimes L_+(t^{\lambda}a^{\lambda})
}
\quad
\xymatrix@C=3em@R=1.5em{
L_+^{(2\lambda)} \ar[r]^{a^{\kappa(\lambda,\lambda)}\cdot R_{a^{2\lambda}}}\ar[d]^{\cong} & L_+(t^{2\lambda}a^{2\lambda}) \ar[d]^{\cong} \\
L_+^{(\lambda)}\otimes L_+^{(\lambda)} \ar[r]^-{R_{a^{\lambda}}\otimes R_{a^{\lambda}}} & L_+(t^{\lambda}a^{\lambda})\otimes L_+(t^{\lambda}a^{\lambda})
}
$$
where vertical arrows witness the multiplicativity of $\cal L_+^{(\lambda)}$. The first diagram commutes because $a_*$ defines a group homomorphism on $E(k[t,t^{-1}])$. The second diagram commutes (note the factor $a^{\kappa(\lambda,\lambda)}$) because it calculates the commutator $\op{comm}(a^{\lambda}, t^{\lambda}) \in \mathbf K_2(k[t,t^{-1}])$, whose residue identifies with $a^{\kappa(\lambda,\lambda)}$.

Now, tracing through the horizontal arrows gives rise to the identity $a^{q_1(2\lambda)-\kappa(\lambda,\lambda)} = a^{2\cdot q_1(\lambda)}$ in $k^{\times}$. Since the same calculation is valid for any field extension $k\subset k'$, we obtain \eqref{eq-quadratic-calc-1}.
\end{proof}

\begin{proof}[Proof of \eqref{eq-quadratic-calc-2}]
This follows from the functoriality of $E(k[t,t^{-1}])$ with respect to the double covering map $\op{sq}(t)= t^2$ on $k[t,t^{-1}]$. Note that $\op{sq}_* : E(k[t,t^{-1}])\rightarrow E(k[t,t^{-1}])$ induces a \emph{quadratic} map of $k^{\times}$-torsors\footnote{i.e., the $k^{\times}$-action on the two lines intertwines $k^{\times}\rightarrow k^{\times}$, $a\leadsto a^2$.}:
$$
\op{sq}_* : L_+^{(\lambda)} \rightarrow L^{(2\lambda)}_+.
$$
On the other hand, we have the following commutative diagrams:
$$
\xymatrix@C=1.5em@R=1.5em{
L_+^{(\lambda)} \ar[d]^{\op{sq}_*}\ar[r]^-{(a^2)_*} & L_+(t^{\lambda}a^{2\lambda}) \ar[d]^{\op{sq}_*} \\
L_+^{(2\lambda)} \ar[r]^-{a_*} & L_+(t^{2\lambda}a^{2\lambda})
}
\quad
\xymatrix@C=2em@R=1.5em{
L_+^{(\lambda)} \ar[d]^{\op{sq}_*}\ar[r]^-{R_{a^{2\lambda}}} & L_+(t^{\lambda}a^{2\lambda}) \ar[d]^{\op{sq}_*} \\
L_+^{(2\lambda)} \ar[r]^-{R_{a^{2\lambda}}} & L_+(t^{2\lambda}a^{2\lambda})
}
$$
The first diagram commutes tautologically. The second diagram commutes because $a^{2\lambda}$ belongs to the subgroup $T(k)\hookrightarrow T(k[t,t^{-1}])$, and we may first lift $a^{2\lambda}$ to $E(k)$ so that its image in $E(k[t,t^{-1}])$ is fixed by the automorphism $\op{sq}_*$. Tracing through the horizontal maps and using the quadraticity of vertical maps, we find $a^{4\cdot q_1(\lambda)}=a^{q_1(2\lambda)}$ in $k^{\times}$. Again because the same calculation is valid for any field extension $k\subset k'$, we obtain \eqref{eq-quadratic-calc-2}.
\end{proof}

\qed(Lemma \ref{lem-torus-compatible})

\subsection{Compatibility: general case}
\label{sec-comp-general}

\subsubsection{}
\label{sec-pic-to-form} We now return to the general case of a reductive group $G$. Appealing to the equivalence \eqref{eq-GrT-to-theta}, we obtain a functor:
\begin{equation}
\label{eq-pic-to-form}
\mathbf{Pic}^{\op{fact}}(\op{Gr}_G)\xrightarrow{\op{res}} \mathbf{Pic}^{\op{fact}}(\op{Gr}_T) \xrightarrow{\sim} \theta(\Lambda_{T}) \rightarrow Q(\Lambda_{T}, \mathbb Z).
\end{equation}

\begin{prop}
\label{prop-ss-class}
Suppose $G$ is semisimple and simply connected. Then \eqref{eq-pic-to-form} defines an equivalence: $\mathbf{Pic}^{\op{fact}}(\op{Gr}_G) \xrightarrow{\sim} Q(\Lambda_{T}, \mathbb Z)^W$.
\end{prop}

In this subsection, we will first prove Proposition \ref{prop-ss-class}, and then use it to deduce the general compatibility result between Gaitsgory functor $\Phi_G$ and the Brylinski-Deligne classification.

\subsubsection{}
We use the notation $\mathbf{Pic}^e(\op{Gr}_G)$ to denote the Picard groupoid of line bundles on $\op{Gr}_G$ together with a rigidification at the unit section $e : \Ran(X)\hookrightarrow\op{Gr}_G$; the notation $\mathbf{Pic}^e(\op{Gr}_{G,X^I})$ carries an analogous meaning. Since factorization line bundles on $\op{Ran}(X)$ are canonically trivial (c.f.~Lemma \ref{lem-classify-Gr-comb}), we have a forgetful functor $\mathbf{Pic}^{\op{fact}}(\op{Gr}_G)\rightarrow\mathbf{Pic}^e(\op{Gr}_G)$.

\subsubsection{}
We first prove Proposition \ref{prop-ss-class} in the case where $G$ is \emph{simple} and simply connected. We note that in this case, the abelian group $Q(\Lambda_T,\mathbb Z)^W$ is isomorphic to $\mathbb Z$, where a generator is given by the \emph{minimal} $W$-invariant quadratic form $q_{\op{min}}$, uniquely specified by the property that $q(\alpha)=1$ for any short coroot $\alpha$.

We fix a point $x\in X$. The calculation of Picard schemes $\mathbf{Pic}^e(\op{Gr}_{G,X^I})$ in \cite[\S3.4]{Zh16} proves that there are isomorphisms:
\begin{equation}
\label{eq-pic-pullback-to-point}
\mathbf{Pic}^{\op{fact}}(\op{Gr}_G) \xrightarrow{\sim} \mathbf{Pic}^e(\op{Gr}_G) \xrightarrow{\sim} \mathbf{Pic}^e(\op{Gr}_{G,x}),
\end{equation}
given by pulling back along $\op{Gr}_{G,x} \hookrightarrow\op{Gr}_G$. On the other hand, the result of G.~Faltings \cite{Fa03} shows that $\mathbf{Pic}^e(\op{Gr}_{G,x})$ is also isomorphic to $\mathbb Z$ (in particular, it is discrete), and the generator of $\mathbf{Pic}^e(\op{Gr}_{G,x})$ is a certain line bundle $\cal L_{\op{min}}$ satisfying the following property:
\begin{enumerate}[(*)]
	\item Let $\cal L_{\det}$ be the determinant line bundle on $\op{Gr}_{G,x}$, whose fiber at an $S$-point $(\cal P_G, \cal P_G\big|_{\overset{\circ}{D}_x}\xrightarrow{\sim}\cal P_G^0)$ is the relative determinant of the lattices $\fr g_{\cal P_G},\fr g_{\cal P_G^0} \subset \fr g(\cal K_x)$. Then there is an isomorphism $(\cal L_{\op{min}})^{\otimes 2\check h} \xrightarrow{\sim} \cal L_{\det}$.
\end{enumerate}

\noindent
In order to show that \eqref{eq-pic-to-form} is an isomorphism onto $Q(\Lambda_T,\mathbb Z)^W$, it suffices to show that for some nonzero integer $d$, the image of $(\cal L_{\op{min}})^{\otimes d}$ (regarded as an element in $\mathbf{Pic}^{\op{fact}}(\op{Gr}_G)$ via \eqref{eq-pic-pullback-to-point}) equals $d\cdot q$. We will prove this statement for $d=2\check h$ by calculating the image of $\cal L_{\det}$.

Note that $\cal L_{\det}$ has a natural factorization structure (c.f.~\cite[\S5.2.1]{GL16}). By tracing through the functors in \eqref{eq-pic-to-form}, we see that its image is the quadratic form $q_{\det}$ whose associated bilinear form $\kappa_{\det}$ equals:
$$
\kappa_{\det}(\lambda,\mu) = \sum_{\check{\alpha}\in\Phi} \langle\lambda,\check{\alpha}\rangle\langle\mu,\check{\alpha}\rangle = \op{Kil}(\lambda,\mu),
$$
where $\op{Kil}$ stands for the Killing form. On the other hand, $\check h$ is defined so that $\op{Kil}=2\check h\cdot\kappa_{\op{min}}$. Thus $q_{\det} = 2\check h\cdot q_{\op{min}}$ as desired.

\subsubsection{} In order to handle the general case, we first note a cohomological vanishing result that will also be useful later. We continue to fix a $k$-point $x\in X$. Recall that for a dominant cocharacter $\lambda\in\Lambda_G^+$, we have the affine Schubert cell $\op{Gr}_{G,x}^{\le\lambda}\hookrightarrow\op{Gr}_{G,x}$ such that $\op{Gr}_{G,x}$ is isomorphic to the infinite union $\underset{\lambda\in\Lambda_T^+}{\op{colim}}\op{Gr}_{G,x}^{\le\lambda}$. When $G$ is semisimple and simply connected, each $\op{Gr}_{G,x}^{\le\lambda}$ is integral.

\begin{lem}
\label{lem-h1-vanishing}
Suppose $G$ is semisimple and simply connected. Then for any $\lambda\in\Lambda_G^+$, we have $\op H^i(\op{Gr}_{G,x}^{\le\lambda},\cal O)=0$ for $i\ge 1$.
\end{lem}
\begin{proof}
Let $I$ denote the Iwahori subgroup of $\cal L_x^+G$ and $\op{Fl}_{G,x}:= \cal L_xG/I$ be the affine flag variety. The $I$-orbits of $\op{Fl}_{G,x}$ are parametrized by the affine Weyl group $W^{\op{aff}}$. Let $\op{Fl}_{G,x}^w$ denote the orbit corresponding to $w\in W^{\op{aff}}$ and $\op{Fl}_{G,x}^{\le w}$ its closure. We note that the projection $\op{Fl}_{G,x}\rightarrow\op{Gr}_{G,x}$ is a flat-locally trivial fiber bundle with typical fiber $G/B$. Furthermore, for any $\lambda\in\Lambda_G^+$, there is a Cartesian square:
$$
\xymatrix@C=1.5em@R=1.5em{
	\op{Fl}_{G,x}^{\le w} \ar@{^{(}->}[r]\ar[d] & \op{Fl}_{G,x} \ar[d] \\
	\op{Gr}_{G,x}^{\le\lambda} \ar@{^{(}->}[r] & \op{Gr}_{G,x}
}
$$
where $w$ is the longest element in the double coset of $\lambda$, after we identify $\Lambda_G^+$ with $W\backslash W^{\op{aff}}/W$. Since $k\xrightarrow{\sim}\op R\Gamma(G/B,\cal O)$, we reduce the proof to showing $k\xrightarrow{\sim}\op R\Gamma(\op{Fl}_{G,x}^{\le w}, \cal O)$.

We now make an argument similar to that for finite dimensional Schubert cells. Namely, for each simple (affine) reflection $s\in W^{\op{aff}}$, we let $P_s := I\cup (IsI)$ denote the corresponding minimal parahoric subgroup. Suppose $w = s_1\cdots s_l$ is an reduced expression. Then we have an \emph{affine} Bott-Samelson resolution:
\begin{equation}
\label{eq-aff-resolution}
\widetilde{\op{Fl}}_G^{\le w} := P_{s_1} \overset{I}{\times} P_{s_2} \overset{I}{\times} \cdots \overset{I}{\times} P_{s_l}/I \rightarrow \op{Fl}_G^{\le w},
\end{equation}
where the $I$-superscripts indicate quotients by anti-diagonal actions. Since each $P_s/I$ is isomorphic to $\mathbb P^1$, the scheme $\widetilde{\op{Fl}}_G^{\le w}$ is an iterated $\mathbb P^1$-bundle. Thus, we reduce to showing that $\cal O_{\widetilde{\op{Fl}}_G^{\le w}}$ has vanishing higher direct image along \eqref{eq-aff-resolution}, and this follows from the same proof as the usual Bott-Samelson resolution, c.f.~\cite[Theorem 2.2.3]{Br04}.
\end{proof}

\begin{rem}
Lemma \ref{lem-h1-vanishing} can be seen as an affine version of the Borel-Weil-Bott theorem and is likely to be known, but the authors could not find a reference.
\end{rem}

\subsubsection{} We now prove Proposition \ref{prop-ss-class} in the general case. Suppose $G$ has simple factors $\{G_j\}_{j\in J}$. It suffices to prove that pulling back along the factors $\op{Gr}_{G_j}\hookrightarrow\op{Gr}_G$ defines an equivalence of Picard groupoids:
\begin{equation}
\label{eq-factor-equiv}
\mathbf{Pic}^{\op{fact}}(\op{Gr}_G) \xrightarrow{\sim} \prod_{j\in J}\mathbf{Pic}^{\op{fact}}(\op{Gr}_{G_j}).
\end{equation}
Note that this morphism fits into a commutative diagram of Picard groupoids:
$$
\xymatrix@C=1.5em@R=1.5em{
	\mathbf{Pic}^{\op{fact}}(\op{Gr}_G) \ar[d]^{\eqref{eq-factor-equiv}}\ar[r] & \mathbf{Pic}^e(\op{Gr}_{G}) \ar[r]^-{(b)}\ar[d]^{(c)} & \mathbf{Pic}^e(\op{Gr}_{G,x}) \ar[d]^{(a)} \\
	\prod_{j\in J}\mathbf{Pic}^{\op{fact}}(\op{Gr}_{G_j}) \ar[r]^-{\sim} & \prod_{j\in J}\mathbf{Pic}^e(\op{Gr}_{G_j}) \ar[r]^-{\sim} & \prod_{j\in J}\mathbf{Pic}^e(\op{Gr}_{G_j,x})
}
$$
where the lower row consists of equivalences, c.f.~\eqref{eq-pic-pullback-to-point}. We note that the cohomological vanishing Lemma \ref{lem-h1-vanishing} for $i=1$ implies that $(a)$ is an equivalence.\footnote{Recall: suppose $X,Y\in\mathbf{Sch}_{/k}$ are connected schemes of finite type with base points, and $X$ is integral, projective with $\op H^1(X,\cal O_X)=0$. Then $\mathbf{Pic}^e(X)\times\mathbf{Pic}^e(Y)\xrightarrow{\sim}\mathbf{Pic}^e(X\times Y)$ (see \cite[Exercise III.12.6]{Ha13}).} That (b) is an equivalence follows from \cite[Lemma 3.4.2]{Zh16} and the proof of \cite[Lemma 3.4.3]{Zh16}. Together, these facts imply that $(c)$ is an equivalence.

\subsubsection{}
Finally, we argue that the left square is Cartesian, which would imply that \eqref{eq-factor-equiv} is an equivalence. Concretely, this means that given a rigidified line bundle $\cal L$ over $\op{Gr}_{G}$ (which passes to $\boxtimes_{j\in J} \cal L_j$ over $\prod_{j\in J}\op{Gr}_{G_j}$ via the equivalence (c)), the datum needed to upgrade it to a factorization structure on $\cal L$:
$$
\varphi : \cal L^{(2)}\big|_{X^2-\Delta} \xrightarrow{\sim} \cal L^{(1)}\boxtimes \cal L^{(1)}
$$
is equivalent to that of factorization structures $\varphi_j$ on each $\cal L_j$. We note that the collection $\{\varphi_j\}_{j\in J}$ defines a factorization structure $\boxtimes_{j\in J}\varphi_j$ on $\cal L$ and conversely a factorization structure $\varphi$ on $\cal L$ defines $\varphi_j$ by restriction to the $j$th unit section $X^2\underset{X^2}{\times}\cdots\underset{X^2}{\times}\op{Gr}_{G_j,X^2}\underset{X^2}{\times}\cdots\underset{X^2}{\times}X^2\hookrightarrow\op{Gr}_{G,X^2}$. Thus it remains to show:

\begin{claim}
Any $\cal L\in\mathbf{Pic}^e(\op{Gr}_{G})$ has at most one factorization structure compatible with its rigidification.
\end{claim}

\noindent
Indeed, any two such factorization structures differ by an automorphism $\beta$ of $\cal L^{(2)}\big|_{X^2-\Delta}$ that restricts to identity along the unit section. Since $\op{Gr}_{G,X^2}\big|_{X^2-\Delta}$ is an ind-integral ind-scheme over $X^2-\Delta$, it suffices to show that $\beta$ becomes the identity after restricting to the fibers at $k$-points of $X^2-\Delta$. The latter follows from the discreteness of $\mathbf{Pic}^e(\op{Gr}_{G,x}\times\op{Gr}_{G,y})$, which in turn follows from that of $\mathbf{Pic}^e(\op{Gr}_{G,x})$ and Lemma \ref{lem-h1-vanishing}. \qed(Proposition \ref{prop-ss-class})

\subsubsection{}

For a semisimple and simply connected group $G$, we obtain a map:
$$
Q(\Lambda_{T},\mathbb Z)^{W}\rightarrow\theta(\Lambda_{T})
$$
by first lifting an element of $Q(\Lambda_T,\mathbb Z)^W$ to $\mathbf{Pic}^{\op{fact}}(\op{Gr}_G)$ using the isomorphism of Proposition \ref{prop-ss-class}, and then mapping to $\theta(\Lambda_T)$. By Lemma \ref{lem-torus-compatible}, the above functor identifies with \eqref{eq-ss-theta}.

\subsubsection{} Recall the Picard groupoid $\theta_G(\Lambda_{T})$ of \S\ref{sec-brde}. We will define a functor:
\begin{equation}
\label{eq-class-general}
\mathbf{Pic}^{\op{fact}}(\op{Gr}_G) \rightarrow \theta_G(\Lambda_{T})
\end{equation}
Given $\cal L\in\mathbf{Pic}^{\op{fact}}(\op{Gr}_G)$, we will construct a theta datum $(q, \cal L^{(\lambda)}, c_{\lambda,\mu})$ for $\Lambda_{T}$ as well as an isomorphism $\varphi$ of two corresponding theta data for $\Lambda_{\widetilde T_{\op{der}}}$.

Indeed, $(q,\cal L^{(\lambda)}, c_{\lambda,\mu})$ is the image of $\cal L$ under the first two maps of \eqref{eq-pic-to-form}. On the other hand, $\cal L$ restricts to a factorization line bundle on $\op{Gr}_{\widetilde G_{\op{der}}}$; under the same two maps, we obtain a theta datum $(q\big|_{\Lambda_{\widetilde T_{\op{der}}}}, \widetilde{\cal L}^{(\lambda)}, \widetilde c_{\lambda,\mu})$. By \S\ref{sec-pic-to-form}, this is the theta datum associated to $q\big|_{\Lambda_{\widetilde T_{\op{der}}}}$ under \eqref{eq-ss-theta}. Therefore, we obtain $\varphi$ from the commutativity datum of the diagram:
$$
\xymatrix@C=1.5em@R=1.5em{
	\mathbf{Pic}^{\op{fact}}(\op{Gr}_{G}) \ar[r]^-{\op{res}}\ar[d] & \mathbf{Pic}^{\op{fact}}(\op{Gr}_T) \ar[r]^-{\sim}\ar[d] & \theta(\Lambda_{T}) \ar[d] \\
	\mathbf{Pic}^{\op{fact}}(\op{Gr}_{\widetilde G_{\op{der}}})\ar[r]^-{\op{res}}  & \mathbf{Pic}^{\op{fact}}(\op{Gr}_{\widetilde T_{\op{der}}}) \ar[r]^-{\sim} & \theta(\Lambda_{\widetilde T_{\op{der}}}).
}
$$

\subsubsection{} We now state the main compatibility result, generalizing Lemma \ref{lem-torus-compatible}:

\begin{prop}
\label{prop-full-compatible}
The following diagram of Picard groupoids commutes functorially in $G$:
\begin{equation}
\label{eq-compatible-main}
\xymatrix@C=-1em@R=1.5em{
	\mathbf{CExt}(G, \mathbf K_2) \ar[rr]^-{\Phi_G}\ar[dr]_{\Phi_{\op{BD}}} & & \mathbf{Pic}^{\op{fact}}(\op{Gr}_G) \ar[dl]^{\eqref{eq-class-general}} \\
	& \theta_G(\Lambda_{T})
}
\end{equation}
\end{prop}
\begin{proof}
Given a central extension of $G$ by $\mathbf K_2$, we have to construct an isomorphism between two elements of $\theta(\Lambda_{T})$ and check that it respects the isomorphism denoted by $\varphi$. The isomorphism comes from the commutativity datum of Lemma \ref{lem-torus-compatible}, and the required compatibility follows from its functoriality with respect to the map of tori $\widetilde T_{\op{der}}\rightarrow T$.
\end{proof}

\medskip

\section{The main theorem}
\label{sec-main}

This section is devoted to the proof that Gaitsgory's functor $\Phi_G$ is an equivalence of categories. We assume $\tn{char}(k) \nmid N_G$ so that the functor $\Phi_G$ is well-defined.

\subsection{Statement and reduction}

\subsubsection{} Let us first state the main theorem of the paper.

\begin{thm}
\label{thm-main}
Suppose $\tn{char}(k)\nmid N_G$. Then the functor $\Phi_G$ \eqref{eq-ga-functor-text} is an equivalence of Picard groupoids.
\end{thm}

Using the commutativity of \eqref{eq-compatible-main} and the fact that $\Phi_{\op{BD}}$ is an equivalence, we have already obtained some special cases of Theorem \ref{thm-main}:
\begin{enumerate}[(a)]
	\item the case $G=T$ is a torus follows from Proposition \ref{prop-classify-torus-case}, as $\theta_G(\Lambda_{T})$ becomes $\theta(\Lambda_{T})$;
	\item the case $G$ semisimple, simply connected follows from Proposition \ref{prop-ss-class}, as $\theta_G(\Lambda_{T})$ becomes the (discrete) abelian group $Q(\Lambda_{T}, \mathbb Z)^{W}$.
\end{enumerate}

\subsubsection{} We now perform a reduction of Theorem \ref{thm-main} to the case where $G_{\op{der}}$ is simply connected. Choose an exact sequence of groups:
\begin{equation}
\label{eq-z-ext}
1 \rightarrow T_2 \rightarrow\widetilde G \rightarrow G\rightarrow 1,
\end{equation}
where $T_2$ is a torus, and $\widetilde G$ is a reductive group whose derived subgroup is simply connected. The sequence \eqref{eq-z-ext} is called a \emph{$z$-extension}, c.f.~\cite[Proposition 3.1]{MS82}. Consider the simplicial system $\widetilde G\times T_2^{\bullet}$, where the $n$th simplex is given by $\widetilde G\times T_2^{\times n}$ and the boundary maps are multiplications. Since $T_2$ is central in $\widetilde G$, these multiplication maps define morphisms of algebraic groups. As a consequence, we obtain a simplicial system of prestacks $\op{Gr}_{\widetilde G\times T_2^{\bullet}}$ over $\Ran(X)$. Appealing to \cite[Corollary 5.2.7]{Ga18}, the Picard groupoid $\mathbf{Pic}^{\op{fact}}(\op{Gr}_G)$ identifies with the limit of the co-simplicial system $\mathbf{Pic}^{\op{fact}}(\op{Gr}_{\widetilde G\times T_2^{\bullet}})$.

\begin{rem}
The cited result follows from $h$-descent of line bundles in the context of \emph{derived} schemes. A proof is given there using the theory of ind-cogerent sheaves, which has only been developped in the context $\tn{char}(k) = 0$. However, one can avoid it by using \cite[\S4]{HLP14} instead.
\end{rem}

\begin{lem}
The canonical map of Picard groupoids is an equivalence:
$$
\mathbf{CExt}(G, \mathbf{K}_2)\xrightarrow{\sim} \lim\mathbf{CExt}(\widetilde G\times T_2^{\bullet}, \mathbf K_2).
$$
\end{lem}
\begin{proof}
We argue that the Picard groupoid of (not necessarily central) extensions $\mathbf{Ext}(G,\mathbf K_2)$ maps isomorphically to $\lim\mathbf{Ext}(\widetilde G\times T_2^{\bullet}, \mathbf K_2)$; the result would follow since a $\mathbf K_2$-extension of $G$ is central if and only if its pullback to each $\widetilde G\times T_2^{\bullet}$ is central.

Since $\mathbf{Ext}(G,\mathbf K_2)$ identifies with homomorphisms from $G$ to $\op B\mathbf K_2$, it suffices to show that $G$ identifies with $\op{colim}(\widetilde G\times T_2^{\bullet})$ in the category of Zariski sheaves of groups (in spaces). This in turn follows from:
\begin{enumerate}[(a)]
	\item the forgetful functor from Zariski sheaves of groups to plain Zariski sheaves is conservative and commutes with geometric realizations;
	\item $G$ identifies with $\op{colim}(\widetilde G\times T_2^{\bullet})$ in the category of plain Zariski sheaves, since every $T_2$-torsor is Zariski-locally trivial (Hilbert 90).\qedhere
\end{enumerate}
\end{proof}

In other words, Theorem \ref{thm-main} for $G$ follows from the same result for each $\widetilde G\times T_2^{\bullet}$. In proving Theorem \ref{thm-main}, we may thus assume that $G_{\op{der}}$ is simply connected.

\subsection{Proof of Theorem \ref{thm-main} for $G_{\op{der}}$ simply connected}

\subsubsection{} We now prove Theorem \ref{thm-main} in the case that $G_{\op{der}}$ is simply connected. Let $T_1:=G/G_{\op{der}}$. Then the fiber of $\theta_G(\Lambda_T)\rightarrow Q(\Lambda_{T_{\op{der}}}, \mathbb Z)^W$ identifies with $\theta(\Lambda_{T_1})$. Let $\mathbf{Pic}^{\op{fact}}_{q_{\op{der}}=0}(\op{Gr}_G)$ be the full subgroupoid of $\mathbf{Pic}^{\op{fact}}(\op{Gr}_G)$, consisting of objects whose images vanish under the following composition:
$$
\mathbf{Pic}^{\op{fact}}(\op{Gr}_G) \rightarrow \mathbf{Pic}^{\op{fact}}(\op{Gr}_{G_{\op{der}}}) \xrightarrow{\eqref{eq-pic-to-form}} Q(\Lambda_{T_{\op{der}}},\mathbb Z).
$$

We then have a commutative diagram of Picard groupoids:
$$
\xymatrix@C=1.5em@R=1.5em{
  & & \mathbf{CExt}(G ;\mathbf K_2) \ar@/^1pc/[ddl]^{\Phi_{\op{BD}}}_{\cong}\ar@/_1pc/[ld]_-{\Phi_G} \\
 \mathbf{Pic}^{\op{fact}}_{q_{\op{der}}=0}(\op{Gr}_G) \ar@{^{(}->}[r]\ar[d] & \mathbf{Pic}^{\op{fact}}(\op{Gr}_G) \ar[d]_{\eqref{eq-class-general}} & \\
 \theta(\Lambda_{T_1}) \ar@{^{(}->}[r] & \theta_{G}(\Lambda_T) \ar[r] &  Q(\Lambda_{T_{\op{der}}}, \mathbb Z)^W.
}
$$
Inspecting this diagram, we see that it suffices to show that the first vertical map:
\begin{equation}
\label{eq-class-qder}
\mathbf{Pic}^{\op{fact}}_{q_{\op{der}}=0}(\op{Gr}_G) \rightarrow \theta(\Lambda_{T_1})
\end{equation}
is an equivalence.

\subsubsection{} Consider the projection $\fr p : \op{Gr}_G \rightarrow \op{Gr}_{T_1}$. It defines a pullback functor
\begin{equation}
\label{eq-proj-to-t1}
\fr p^* : \mathbf{Pic}^{\op{fact}}(\op{Gr}_{T_1})\rightarrow \mathbf{Pic}^{\op{fact}}_{q_{\op{der}}=0}(\op{Gr}_G)
\end{equation}
such that the composition:
$$
\mathbf{Pic}^{\op{fact}}(\op{Gr}_{T_1})\xrightarrow{\fr p^*} \mathbf{Pic}^{\op{fact}}_{q_{\op{der}}=0}(\op{Gr}_G) \xrightarrow{\eqref{eq-class-qder}} \theta(\Lambda_{T_1})
$$
canonically identifies with the equivalence \eqref{eq-GrT-to-theta}. It therefore suffices to show that \eqref{eq-proj-to-t1} is an equivalence.

\subsubsection{}
We note that \eqref{eq-proj-to-t1} factors through the full subcategory
	\begin{equation}
	\label{eq-pic-natural}
	\mathbf{Pic}^{\op{fact}}_{\natural}(\op{Gr}_G)\hookrightarrow \mathbf{Pic}^{\op{fact}}_{q_{\op{der}}=0}(\op{Gr}_G)
	\end{equation}
	of factorization line bundles on $\op{Gr}_G$ which are trivial along fibers of $\fr p$ over $k$-points. In the rest of this subsection, we shall show that
\begin{enumerate}[(a)]
	\item the containment \eqref{eq-pic-natural} is an equivalence.
	\item pullback along $\fr p$ defines an equivalence
	\begin{equation}
	\label{eq-pullback-equiv}
	\mathbf{Pic}^{\op{fact}}(\op{Gr}_{T_1})\rightarrow \mathbf{Pic}^{\op{fact}}_{\natural}(\op{Gr}_G).
	\end{equation}
\end{enumerate}
The combination of these two statements will imply Theorem \ref{thm-main}.

\subsubsection{} In order to prove the above statements, we first study the geometric properties of the projection $\fr p$.

\begin{lem}
\label{lem-proj-to-t1}
The map $\fr p$ realizes $\op{Gr}_G$ as an \'etale locally trivial $\op{Gr}_{G_{\op{der}}}$-bundle over $\op{Gr}_{T_1}$.
\end{lem}

In other words, for every affine scheme $S\rightarrow\op{Gr}_{T_1}$, there is an \'etale cover $\widetilde S\rightarrow S$ and an isomorphism $\op{Gr}_G\underset{\op{Gr}_{T_1}}{\times}\widetilde S\xrightarrow{\sim} \op{Gr}_{G_{\op{der}}}\underset{\Ran(X)}{\times}\widetilde S$.

\begin{proof}[Proof of Lemma \ref{lem-proj-to-t1}]
We first show that $G\rightarrow T_1$ splits. Indeed, the maximal (split) torus $T\subset G$ surjects onto $T_1$, so it suffices to show that the kernel $T\cap G_{\op{der}}$ is connected. The latter follows since $T\cap G_{\op{der}}$ is a maximal torus of $G_{\op{der}}$.

Given an $S$-point $S\xrightarrow{\gamma}\op{Gr}_{T_1}$, we denote by $S\xrightarrow{\gamma_0}\op{Gr}_{T_1}$ the ``neutral point'' corresponding to $\gamma$, i.e., the composition $S \xrightarrow{\gamma}\op{Gr}_{T_1}\xrightarrow{\pi}\op{Ran}(X) \hookrightarrow\op{Gr}_{T_1}$. Since $\op{Gr}_{G}\underset{\op{Gr}_{T_1},\gamma_0}{\times}S$ identifies with $\op{Gr}_{\widetilde G_{\op{der}}}\underset{\Ran(X)}{\times}S$, it suffices to produce an isomorphism:
\begin{equation}
\label{eq-noncan-isom}
\op{Gr}_{G} \underset{\op{Gr}_{T_1},\gamma}{\times}\widetilde S \xrightarrow{\sim} \op{Gr}_{G} \underset{\op{Gr}_{T_1},\gamma_0}{\times}\widetilde S
\end{equation}
after passing to some \'etale cover $\widetilde S\rightarrow S$.

We choose $\widetilde S\rightarrow S$ such that the elements $\gamma,\gamma_0\in\op{Maps}_{/\op{Ran}(X)}(\widetilde S, \op{Gr}_{T_1})$ differ by the action of some $\alpha\in\op{Maps}_{/\op{Ran}(X)}(\widetilde S, \cal LT_1)$ (this is possible, for example, by lifting $S\rightarrow\op{Gr}_{T_1}$ to $\widetilde S\rightarrow\cal LT_1$). The above discussion shows that we have a splitting of the canonical projection $\cal LG\rightarrow\cal LT_1$. Hence $\alpha$ can be lifted to an element $\widetilde{\alpha}\in\op{Maps}_{/\op{Ran}(X)}(\widetilde S, \cal LG)$. The equivariance property of $\fr p$ shows that the following diagram commutes:
$$
\xymatrix@R=1.5em{
\op{Gr}_{G} \underset{\op{Ran}(X)}{\times} \widetilde S \ar[r]^{\op{act}_{\widetilde{\alpha}}}\ar[d] & \op{Gr}_{G} \underset{\op{Ran}(X)}{\times} \widetilde S \ar[d] \\
\op{Gr}_{T_1} \underset{\op{Ran}(X)}{\times} \widetilde S \ar[r]^{\op{act}_{\alpha}} & \op{Gr}_{T_1} \underset{\op{Ran}(X)}{\times} \widetilde S
}
$$
Since $\op{act}_{\alpha}$ transforms the section $\gamma : \widetilde S\rightarrow\op{Gr}_{T_1}\underset{\op{Ran}(X)}{\times}\widetilde S$ to $\gamma_0$, we obtain the required isomorphism \eqref{eq-noncan-isom} as $\op{act}_{\widetilde{\alpha}}\underset{\op{act}_{\alpha}}{\times}\op{id}_{\widetilde S}$.
\end{proof}

\subsubsection{Proof of $(a)$} We now show that every $\cal L\in\mathbf{Pic}^{\op{fact}}_{q_{\op{der}}=0}(\op{Gr}_G)$ is fiberwise trivial along the projection $\fr p: \op{Gr}_G \rightarrow\op{Gr}_{T_1}$. Since the question concerns only points on $\op{Gr}_{T_1}$, it suffices to show that the base change of $\cal L$ to the subscheme $X^{(\lambda_1,\cdots,\lambda_{|I|})}\hookrightarrow\op{Gr}_{T_1, X^I}$\footnote{Recall that for an $I$-family of co-characters $\lambda^{(I)}=(\lambda_1,\cdots,\lambda_{|I|})$, there is a closed immersion $X^I\hookrightarrow\op{Gr}_{T_1,X^I}$ whose image we call $X^{(\lambda_1,\cdots,\lambda_{|I|})}$.} is fiberwise trivial.

We write $\underline{\mathbf P}^{(\lambda^I)}$ for the \'etale sheaf of relative Picard group of $\op{Gr}_{G,X^I}\underset{\op{Gr}_{T_1,X^I}}{\times} X^{(\lambda_1,\cdots,\lambda_{|I|})}$ over $X^{(\lambda_1,\cdots,\lambda_{|I|})}$, i.e., it associates to every \'etale map $V\rightarrow X^{(\lambda_1,\cdots,\lambda_{|I|})}$ the abelian group $\mathbf{Pic}(\op{Gr}_{G,X^n}\underset{\op{Gr}_{T_1,X^n}}{\times}V)/\mathbf{Pic}(V)$. Thus $\cal L$ defines a global section $l^{(\lambda^I)}$ of $\underline{\mathbf P}^{(\lambda^I)}$ for every $n$-tuple $\lambda^I$. The goal is to show that all $l^{(\lambda^I)}$ vanish.

\subsubsection{} Recall the computation of the \'etale sheaf of relative Picard groups $\underline{\mathbf{Pic}}(\op{Gr}_{G_{\op{der}}, X^I}/X^I)$ in \cite[\S3.4]{Zh16}. It fits into an exact sequence of sheaves of abelian groups over $X^I$:
$$
0 \rightarrow \underline{\mathbf{Pic}}(\op{Gr}_{G_{\op{der}}, X^I}/X^I) \rightarrow \boxtimes_{i\in I} \underline A_X \rightarrow \bigoplus_{|J|=|I|-1} (\Delta_{I\twoheadrightarrow J})_*\boxtimes_{j\in J}\underline A_X.
$$
Here, $A$ denotes the abelian group $\mathbb Z^{\times\op{rank}( G_{\op{der}})}$, and $\underline A_X$ is its associated constant sheaf of groups over $X$. Lemma \ref{lem-proj-to-t1} shows that the sheaf $\underline{\mathbf P}^{(\lambda^I)}$ is \'etale locally isomorphic to $\underline{\mathbf{Pic}}(\op{Gr}_{G_{\op{der}},X^I}/X^I)$ under the identification $X^{(\lambda_1,\cdots,\lambda_{|I|})}\xrightarrow{\sim} X^I$. We note a simple Lemma:

\begin{lem}
\label{lem-contamination}
Let $Y$ be a connected, Noetherian scheme and $\cal F$ be an \'etale sheaf on $Y$. Suppose furthermore that $\cal F$ is \'etale locally isomorphic to a subsheaf of a constant sheaf. Then a section $s\in\Gamma(Y, \cal F)$ vanishes if and only if it does so over some \'etale open $V\rightarrow Y$. 
\end{lem}
\begin{proof}
One can pick finitely many \'etale maps $V_i\rightarrow Y$ ($i\in I$) so that:
\begin{enumerate}[(a)]
	\item each $V_i$ is connected;
	\item $\cal F\big|_{V_i}$ is isomorphic to a subsheaf of a constant sheaf;
	\item the images $U_i$ of $V_i$ collectively cover $Y$.
\end{enumerate}
We induct on the cardinality of $I$ over all connected, Noetherian schemes admitting such a cover; the base case $I=\emptyset$ is trivial. The image $U$ of $V\rightarrow Y$ must intersect some $U_i$. The condition (b) implies that the restriction $s_i\in\Gamma(U_i,\cal F)$ vanishes. Now, let $\overset{\circ}{Y}:=\bigcup_{j\neq i}U_i$. It is \emph{not} necessarily connected. However, the fact that $Y$ is connected shows that $U_i$ intersects every connected component of $\overset{\circ}{Y}$. We apply the induction hypothesis to each connected component of $\overset{\circ}{Y}$ to conclude that $s$ vanishes.
\end{proof}

\subsubsection{} Our proof that each $l^{(\lambda^I)}$ vanishes now proceeds as follows:

\smallskip

\textbf{Step 1:} $l^{(0)}=0$. Indeed, since line bundles on $\op{Gr}_{G_{\op{der}}, X}$ are classified by the quadratic form $q_{\op{der}}$, we see that $\cal L$ is trivialized when pulled back along $\op{Gr}_{G_{\op{der}}, X}\rightarrow\op{Gr}_{G,X}$. On the other hand, $\op{Gr}_{G_{\op{der}}, X}$ appears as the fiber of $\fr p$ along the unit map $X\hookrightarrow\op{Gr}_{T_1}$. Hence $l^{(0)}=0$.

\smallskip

\textbf{Step 2:} $l^{(\lambda)}=0$ for all $\lambda\in\Lambda_{T_1}$. Consider the section $l^{(\lambda,-\lambda)}$ of $\underline{\mathbf P}^{(\lambda,-\lambda)}$. It is represented by some line bundle $\cal L^{(\lambda,-\lambda)}$ over $\op{Gr}_{G,X^2}\underset{\op{Gr}_{T_1,X^2}}{\times}X^{(\lambda,-\lambda)}$. We know from Step 1 that the restriction of $\cal L^{(\lambda,-\lambda)}$ to the diagonal comes from the base $X^{(0)}\hookrightarrow X^{(\lambda,-\lambda)}$. Hence, over an \'etale neighborhood of $X^{(0)}$, the section $l^{(\lambda,-\lambda)}$ has to vanish by the identification of $\underline{\mathbf P}^{(\lambda,-\lambda)}$ with $\underline{\mathbf{Pic}}(\op{Gr}_{G_{\op{der}}, X^2}/X^2)$. We then apply Lemma \ref{lem-contamination} to conclude that $l^{(\lambda,-\lambda)}$ vanishes.

Now, under the identification of $\underline{\mathbf P}^{(\lambda,-\lambda)}$ with $\underline{\mathbf P}^{(\lambda)}\boxtimes\underline{\mathbf P}^{(-\lambda)}$ away from the diagonal, the section $l^{(\lambda,-\lambda)}$ passes to $l^{(\lambda)}\boxtimes l^{(-\lambda)}$. The fact that $l^{(\lambda,-\lambda)}=0$ now implies that $l^{(\lambda)}$ (and $l^{(-\lambda)}$) vanishes.

\smallskip

\textbf{Step 3:} $l^{(\lambda^I)}=0$ for all $I$-tuple $\lambda^I$. When the cardinality of $I$ is at least $2$, we may use the factorization property of $l^{(\lambda^I)}$ to see that $l^{(\lambda^I)}$ vanishes away from the union of the diagonals in $X^{(\lambda_1,\cdots,\lambda_{|I|})}$. Hence by Lemma \ref{lem-contamination} again we have $l^{(\lambda^I)}=0$.

This finishes the proof that \eqref{eq-pic-natural} is an equivalence.

\subsubsection{Proof of $(b)$} We first recall some standard results.

\begin{lem}
\label{lem-ss-good}
Suppose $\widetilde G$ is semisimple and simply connected. Then the morphism $\op{Gr}_{\widetilde G}\rightarrow\op{Ran}(X)$ has the property that for every affine scheme $S\rightarrow\op{Ran}(X)$, we have a presentation
$$
\op{Gr}_{\widetilde G} \underset{\op{Ran}(X)}{\times} S \xrightarrow{\sim} \underset{i}{\op{colim}}\,Y_i
$$
where each $Y_i$ is a scheme of finite type over $S$, satisfying:
\begin{enumerate}[(a)]
	\item $Y_i$ is proper and faithfully flat over $S$;
	\item The fiber $(Y_i)_s$ at every $k$-point $s\in S$ is connected and $\op H^1((Y_i)_s, \cal O)\cong 0$.
\end{enumerate}
\end{lem}
\begin{proof}
Since each $S\rightarrow\op{Ran}(X)$ factors through some $X^I$, it suffices to produce such a presentation for $\op{Gr}_{\widetilde G,X^I}$. For each $I$-tuple $\underline{\lambda}$ of elements of $\Lambda_G^+$, we may consider the Schubert variety $\op{Gr}_{\widetilde G,X^I}^{\le\underline{\lambda}}$ which is proper, surjective over $X^I$. The flatness is proved in \cite[\S1.2]{Zh09} for $I=\{1,2\}$ and the general case is similar. The property (b) of its fibers is a special case of Lemma \ref{lem-h1-vanishing}.
\end{proof}

\begin{rem}
Lemma \ref{lem-ss-good}(b) fails for non-semisimple groups, since $\op{Gr}_G$ may not be ind-reduced. We do \emph{not} know whether the flatness in part (a) holds more generally.
\end{rem}

\subsubsection{}
\label{sec-pic-push-sch} Suppose $p:X\rightarrow Y$ is a morphism of finite type schemes over $k$\footnote{Recall that $k$ is assumed to be algebraically closed.} such that
\begin{enumerate}[(a)]
	\item $p$ is proper and faithfully flat;
	\item its fiber $X_y$ at every $k$-point $y\in Y$ is connected and $\op H^1(X_y, \cal O)=0$.
\end{enumerate}

\begin{lem}
\label{lem-line-bundle-push}
Let $\cal L$ be a line bundle on $X$. Under the above hypotheses on $p:X\rightarrow Y$, the following are equivalent:
\begin{enumerate}[(a)]
	\item $\cal L$ is trivial along the fibers of $p$;
	\item $p_*\cal L$ is a line bundle over $Y$, and the canonical map $p^*p_*\cal L\rightarrow\cal L$ is an isomorphism.
\end{enumerate}
\end{lem}
\begin{proof}
We use the formulation of the ``cohomology and base change'' theorem in \cite[28.1.6]{Va}. The fiberwise triviality of $\cal L$, together with the vanishing of $\op H^1(X_y, \cal O_{X_y})$, shows that the canonical map:
\begin{equation}
\label{eq-coh-base-change}
\op R^1p_*\cal L\big|_y \rightarrow \op H^1(X_y, \cal L\big|_{X_y})
\end{equation}
is surjective, for any $k$-point $y\in Y$. Hence part (i) of \emph{loc.cit.} applies and we see that that \eqref{eq-coh-base-change} is an isomorphism. Since $\op R^1p_*\cal L$ is coherent, it must vanish. In particular, part (ii) of \emph{loc.cit.} applies and shows that the canonical map $p_*\cal L\big|_y\rightarrow \op H^0(X_y,\cal L\big|_{X_y})$ is surjective. Another application of part (i) then shows that $p_*\cal L$ is locally free near $y$ of rank $h^0(X_y, \cal L\big|_{X_y})=h^0(X_y, \cal O)=1$, i.e., it is a line bundle. The isomorphism $p^*p_*\cal L\xrightarrow{\sim}\cal L$ is then obvious.
\end{proof}

\subsubsection{}
\label{sec-pic-push-indsch} Suppose $p: \cal X\rightarrow Y$ is ind-schematic morphism, represented by morphisms $p_i: X_i\rightarrow Y$ of schemes satisfying the hypothesis of \S\ref{sec-pic-push-sch}. Then $p^* : \mathbf{Pic}(Y) \rightarrow \mathbf{Pic}(\cal X)$ has a partially defined right adjoint:
$$
p_*\cal L := \underset{i}{\op{lim}}\,(p_i)_*\cal L_i,\quad\text{while representing $\cal L$ by the inverse system $\cal L_i\in\mathbf{Pic}(X_i)$}
$$
which is well defined on the full subcategory of $\mathbf{Pic}(\cal X)$ consisting of line bundles which are trivial along the fibers of $p$, and we furthermore have an isomorphism $p^*p_*\cal L\xrightarrow{\sim}\cal L$. For any line bundle $\cal M$ from the base $Y$, it is also clear that $\cal M\xrightarrow{\sim} p_*p^*\cal M$. Hence $p^*$ defines an equivalence from $\mathbf{Pic}(Y)$ to the full subcategory of $\mathbf{Pic}(\cal X)$ consisting of fiberwise trivial line bundles.

\subsubsection{}
The above discussion, together with Lemma \ref{lem-proj-to-t1} and \ref{lem-ss-good}, shows that $\fr p^*$ defines an equivalence $\mathbf{Pic}(\op{Gr}_{T_1})\xrightarrow{\sim} \mathbf{Pic}_{\natural}(\op{Gr}_G)$. To see that this upgrades to an equivalence of factorization line bundles, we simply note that the map $\op{Gr}_G\underset{\op{Ran}(X)}{\times}\op{Gr}_G\rightarrow \op{Gr}_{T_1}\underset{\op{Ran}(X)}{\times}\op{Gr}_{T_1}$ again satisfies the hypothesis of \S\ref{sec-pic-push-indsch} after base change to a scheme. This finishes the proof that \eqref{eq-pullback-equiv} is an equivalence. \qed(Theorem \ref{thm-main})

\end{document}